\documentclass[12pt,psamsfonts]{amsart}
\usepackage{amsmath}
\usepackage{amsthm}
\usepackage{amssymb}
\usepackage{amscd}
\usepackage{amsfonts}
\usepackage{amsbsy}

\def\be{\begin{equation}}
\def\ee{\end{equation}}

\def\bq{\begin{eqnarray}}
\def\eq{\end{eqnarray}}

\def\beq{\begin{eqnarray*}}
\def\eeq{\end{eqnarray*}}

\newtheorem {theorem} {Theorem}[section]
\newtheorem {proposition} [theorem]{Proposition}
\newtheorem {corollary} [theorem]{Corollary}

\newtheorem {conjecture} [theorem]{Conjecture}

\begin{document}

\title[Cartesian approach for constrained mechanical systems .]
{Cartesian approach for constrained mechanical systems.}

\date{}
\dedicatory{}
\maketitle

%\author[ R. Ram\'{\i}rez and N. Sadovskaia]
\centerline{Rafael Ram\'{\i}rez}
\smallskip
\centerline{ Departament d'Enginyeria Inform\`{a}tica i
Matem\`{a}tiques.} \centerline{ Universitat Rovira i Virgili.}
\centerline{ Avinguda dels Pa\"{\i}sos Catalans 26, 43007
Tarragona,  Spain.}
\centerline{E-mail:rafaelorlando.ramirez@urv.cat} \centerline{and}
\smallskip
\centerline{Natalia Sadovskaia}
\smallskip
 \centerline{ Departament de
Matem\`{a}tica Aplicada II} \centerline{ Universitat
Polit\`{e}cnica de Catalunya,} \centerline{ C. Pau Gargallo 5,
08028 Barcelona,
 Spain.}
\centerline{E-mail:natalia.sadovskaia@upc.edu}

\maketitle

\begin{abstract}

In the history of mechanics, there have been two points of view
for studying mechanical systems:  Newtonian and  Cartesian.
\smallskip
According the Descartes point of view, the motion of mechanical
systems is described by the first-order differential equations in
the $N$ dimensional configuration space $\textsc{Q}.$
\smallskip
 In this paper we develop the Cartesian approach for mechanical
systems with constraints which are linear with respect to
velocity.
\end{abstract}



\section{Introduction.}
In "Philosophiae Naturalis Principia Mathematica" (1687), Newton
considers that movements of celestial bodies can be described by
differential equations of the second order. To determine their
trajectory, it is necessary to give the initial position and
velocity. To reduce the equations of motion to the investigation
of a dynamics system it is necessary to double the dimension of
the position space and to introduce the auxiliary phase space.
\smallskip
Descartes in 1644  proposed that the behavior of the celestial
bodies be studied from another point of view.  These ideas were
stated in "Principia Philosophiae" (1644) and in "Discours de la
m\'etode" (1637). According to Descarte the understanding of
cosmology starts from acceptance of the initial chaos, whose
moving elements are ordered according to certain fixed laws and
form the Cosmo. He consider that the Universe is filled with a
tenuous fluid matter (ether), which is constantly in a vortex
motion. This motion moves the largest particle of matter of the
vortex axis, and they subsequently form planets. Then, according
to what Descartes wrote in his "Treatise on Light", "the material
of the Heaven must be rotate the planets not only about the Sun
but also about their own centers...and this will hence form
several small Heavens rotating in the same direction as the great
Heaven."\cite{Koz1}. Thus the equation of motion in the Descartes
theory must be of the first order equation in the configuration
space $ \textsc{Q}$
  \[\dot{\textbf{x}}=\textbf{v}(x,t),\quad
  x\in\textsc{Q}\]
 Descartes gave no principles  for constructing the field $
 \textbf{v}$ for different mechanical systems.
 Hence, to determine the trajectory from Descartes's point of
view it is necessary to give only the initial position.
\smallskip
In the modern scientific literature the study of the Descarte
ideas we can find in the monographic of V.V. Kozlov \cite{Koz1} in
which the author  affirms  that  "solving dynamics problem is
possible inside the configuration space".

\smallskip

In \cite{Ram1,Ram4,Sad} we developed the Cartesian approach for
mechanical systems with three degrees of freedom  and with
constraints linear with respects to velocity. The aim of this
works is to generalized the Cartesian approach for non-holonomic
mechanical system with $N$ degrees of freedom and constraints
which are linear with respect to the velocity.

\smallskip
We shall present briefly the contents of the paper.

\smallskip
In section 2 we prove our  main results (see Theorem \ref{main},
Corollary \ref{main1}, Corollary \ref{main2}, Corollary
\ref{main0} below).

\smallskip

In section 3 Corollary \ref{main1} applied to determine Cartesian and lagrangian approach for
non-holonomic systems with three  degree of freedom.
We illustrate the obtained results to study Chapliguin-Caratheodory's sleigh and  to
 study Suslov's problem for the rigid body around a fixed point.

\smallskip

In section 4 we determine Cartesian and lagrangian approach in
three dimensional Euclidean space.
\smallskip

In section 5 by applying the results of the previous section we
study the integrability of the geodesic flow on the surface.

\smallskip

In section 6 Theorem \ref{main} applied to the study Gantmacher's
system and Rattleback.

\smallskip
In section 7 Corollary \ref{main2} applied to solve the inverse
problem in dynamics.

\smallskip

 For simplicity we shall assume the underlying functions to be of
 class $ \mathcal{C}^\infty,$ although most results remain valid
 under weaker hypotheses.

\smallskip

It is well known that the behavior of constrained Lagrangian system
\[
\langle\textsc{ Q},
L=\dfrac{1}{2}\sum_{j,k=1}^NG_{kj}(x)\dot{x}^j\dot{x}^k-U(x),\quad\sum_{k=1}^N\alpha_{jk}(x)\dot{x}^k=0,\quad
j=1,2,\ldots,M,\rangle
\]
can be described by the differential equations deduced from the
D'Alembert-Lagrange Principle\cite{Koz3,Ram2}
\begin{equation}\label{L21}
\displaystyle\frac{d}{dt}\displaystyle\frac{\partial
T}{\partial\dot{x}^k}-\displaystyle\frac{\partial
T}{\partial{x}^k}=\displaystyle\frac{\partial
U}{\partial{x}^k}+\sum_{j=1}^M\mu_j \alpha_{jk}(x),\quad
k=1,2,\ldots N,
\end{equation}
where $\mu_1,\,\mu_2,\ldots,\mu_M$ are Lagrangian multipliers.

\smallskip

Our main results are the following
\begin{theorem}\label{main}
Let Let $\textsc{Q}$ be a smooth manifold of the dimension $N$
with local coordinates $ x=(x^1,..., x^N)$ and equipped by the
Riemann metric $G=(G_{kj}{(x)})=(G_{kj}) $ and let
\[
{\sigma} =(\textbf{
v}(x),d\textbf{x})=\sum_{j,k=1}^NG_{kj}\dot{x}^jv^k,\] be the
1-form associated to the vector field
\begin{equation}\label{00}
 \textbf{v}=\frac{1}{\Upsilon}\left|
\begin{array}{ccccc}
\Omega_1(\partial_1)&\Omega_1(\partial_2)&\hdots  &
\Omega_1(\partial_N)&0\\
\vdots          & \vdots         & \hdots & \vdots&\vdots          \\
\Omega_{M}(\partial_1)&\Omega_{M}(\partial_2)&\hdots
&\Omega_{M}(\partial_N)&0\\
\Omega_{M+1}(\partial_1)&\Omega_{M+1}(\partial_2)&\hdots
&\Omega_{M+1}(\partial_N)&\lambda_{M+1}\\
\vdots          & \vdots         & \hdots & \vdots &\vdots         \\
\Omega_{N}(\partial_1)&\Omega_{N}(\partial_2)&\hdots
&\Omega_{N}(\partial_N)&\lambda_N\\
\partial_1&\partial_2&\hdots  & \partial_N&0\end{array}
         \right|,
\end{equation}
 where  $\Upsilon={\Omega_1\wedge{\Omega_2}...\wedge{\Omega_N}(\partial_1,\partial_2,...,
\partial_N)},$ \, $\partial_k=\displaystyle{\frac{\partial}{\partial{x^k}}},$
$ \lambda_{j}=\lambda_j(x)$  for $j=M+1,\ldots N$ are arbitrary functions  and $\Omega_k$ for $k=1,\ldots N$ are 1-forms on $\textsc{Q} $
which  we assume that satisfies the following conditions
\begin{itemize}
\item[(i)]
$ \Omega_j$ for $j=1,\ldots,M$ are a given 1-forms: $
\Omega_j=\displaystyle\sum_{k=1}^N\alpha_{jk}dx^k$ where
$\alpha_{jk}= \alpha_{jk}(x)$ are  functions on $ \textsc{Q},$

\item[(ii)] $ \Omega_{k}$ for $k=M+1,\ldots,N$ are arbitrary  1-forms which we choose in such a way that
$\Upsilon\ne{0}.$
\item[(iii)] The 2-form $d\sigma$ admits the development
 $$d\sigma=\dfrac{1}{2}\sum_{j,k=1}^Na_{jk}(x)\Omega_j\wedge\Omega_k$$
  where $A=(a_{jk})$ is a skew symmetric matrix such that$$ H=M^TAM$$ where $H=(d\sigma ( \partial_j,\,\partial_k))$ and $M=(\Omega_j( \partial_k))$ are $N\times N$ matrix.
\item[(iv)] The contraction of the 2-form $d\sigma$ along the vector field $ \textbf{v}$ is such that
\[ \imath_{ \textbf{v} }d\sigma=\sum_{j=1}^N\Lambda_j(x) \Omega_j, \]
 where $\Lambda=\left(\Lambda_1,\,\ldots,\Lambda_N\right)^T$  which we can be calculated as follows
\begin{equation}\label{02}
{{\Lambda}}=A\tilde{\lambda}=M^{-1}\tau,\,
\end{equation}
where  $\tilde{\lambda}=\left(\lambda_1,\ldots,\lambda_N\right)^T$
and
$\tau=\left(\iota_{{\textbf{v}}}d{\sigma}(\partial_1)\,,\ldots,\iota_{{\textbf{v}}}d{\sigma}(\partial_N)\right)^T$

 \end{itemize}

 Then the first order differential system on $ \textsc{Q}$
 \begin{equation}\label{L0}
 \dot{\textbf{x}}=\textbf{v}(x),\quad \mbox{under the conditions},\quad \quad \Lambda_j(x)=0,
   \end{equation}for $ j=M+1,\ldots N,$
 is invariant relationship of the second order differential system
 \begin{equation}\label{L2}
\displaystyle\frac{d}{dt}\displaystyle\frac{\partial
T}{\partial\dot{x}^k}-\displaystyle\frac{\partial
T}{\partial{x}^k}=\displaystyle\frac{\partial
\dfrac{1}{2}||\textbf{v}||^2}{\partial{x}^k}+\sum_{j=1}^M\Lambda_j \alpha_{jk},\quad
k=1,2,\ldots N,
\end{equation}
\end{theorem}

Comparing equations \eqref{L21} and \eqref{L2} we deduce that the
latter can be interpreted as the equations describing the behavior
of non-holonomic mechanical systems under the action of active
forces with potential  $U=\dfrac{1}{2}||\textbf{v}||^2+U_0,\quad
U_0=const$ and under the action of the reactive forces with the
components
$$\left(\sum_{j=1}^M\Lambda_j\alpha_{j1},\,
\sum_{j=1}^M\Lambda_j\alpha_{j2},....,
\sum_{j=1}^M\Lambda_j\alpha_{jN}\right),$$ generated by the
constraints
$\Omega_j(\dot{\textbf{x}})=\displaystyle{\sum_{k=1}^N\alpha_{jk}\dot{x}^k}=0,$
for $ j=1,2,...M.$

\smallskip
Of interest is that the equations $ \Lambda_j=0$ for $j=M+1,\ldots,N$ or, which is the same
\begin{equation}\label{C3}
\Lambda_j=\sum_{k=1}^Na_{jk}\lambda_k=0,\quad a_{jk}=-a_{kj}
\end{equation} for $j=M+1,\ldots,N,$ represent a system of partial differential equations of first order with respect
to the functions $\lambda_k$ for $k=1,\ldots,N.$
\smallskip

\textbf{Definition}

\smallskip

We call the vector field $\textbf{v} $ which generated system \eqref{L0}
{\it Cartesian vector field}. The vector field $ \breve{\textbf{v}}$ we say {\it Cartesian equivalent} if there exist a nonzero
 function $\kappa$ on $\textsc{Q}$ such that $ \kappa\breve{\textbf{v}}$ is Cartesian vector field.
\smallskip

 \textbf{Definition}

\smallskip

Studying the behavior of nonholonomic mechanical systems with constraints linear with respect to the velocity  using the
    equations \eqref{L21},\,\eqref{L0} and  \eqref{L2} we  called {\it Classical,  Decartes } and {\it Lagrangian
     approach} respectively.

\smallskip

\begin{conjecture}\label{c4}
There are solutions of the equations \eqref{C3}  which generate a
Cartesian (or Cartesian equivalent) vector field which completely
describe the behavior of the study constrained Lagrangian system.
\end{conjecture}

\smallskip

This conjecture supports the following facts.

First, in view of Theorem\ref{main} the solutions of \eqref{L0}
are solutions of \eqref{L2}, which is closely linked to the system
\eqref{L21}. Second, the solutions of the equations \eqref{L21}
depend on the $2N-M$ initial conditions. The solutions of
\eqref{L0} depend on $N$ initial conditions and $N-M$ functions
which are solutions of the linear partial differential equations
\eqref{C3}, therefore the solutions of the equations \eqref{C3}
also depend on $N-M$ arbitrary constants. Finally, also contribute
to the strengthening of the conjecture the large number of
applications given below.
\begin{corollary}\label{main1}
Let us suppose that  in Theorem \ref{main} manifold $ \textsc{Q}$
is three dimensional smooth manifold with local coordinates
$\textbf{x}=(x,y,z)$ and the given 1-form is
$\Omega=a_1dx+a_2dy+a_3dz=(\textbf{a},d\textbf{x})$ where
$a_j=a_j(x,y,z)$ for $j=1,2,3$ are  functions on $ \textsc{Q}.$

Denoting by\quad$[\quad\times \quad]$ the vector product in
$\mathbf{R}^3$ and by $\mbox{rot}{\textbf{v}},\,\textbf{a}
,\,\textbf{w}$ the following vectors fields
\[\begin{array}{ll}
\mbox{rot}{\textbf{v}}=&\displaystyle{\dfrac{1}{\sqrt{det{G}}}}\left((\partial_yp_3-\partial_zp_2),\,
(\partial_zp_1-\partial_xp_3),\,
(\partial_xp_2-\partial_yp_1)\right)^T,\vspace{0.20cm}\\

\textbf{w}==&\dfrac{1}{\Upsilon}\left((\lambda_2\Omega_3-
\lambda_3\Omega_2)(\partial_x),(\lambda_2\Omega_3-
\lambda_3\Omega_2)(\partial_y),(\lambda_2\Omega_3-
\lambda_3\Omega_2)(\partial_z)\right)^T.
\end{array}
\]
where $p_k=\displaystyle\sum_{j=1}^3G_{kj}v^j,$ for $k=1,2,3,$
then differential system \eqref{L0} and \eqref{L2} take the form respectively
\begin{equation}\label{L1}
\dot{\textbf{x}}=[\textbf{a}\times \textbf{w}]=\textbf{v}(x),\quad
\left(\textbf{a},\mbox{rot}[\textbf{a}\times \textbf{w}]\right)=0.
\end{equation}
\begin{equation}\label{LLL1}
\begin{array}{cc}
\displaystyle\frac{d}{dt}\displaystyle\frac{\partial,
T}{\partial\dot{\textbf{x}}}-\displaystyle\frac{\partial
T}{\partial{\textbf{x}}}=\displaystyle\frac{\partial
\dfrac{1}{2}||\textbf{v}||^2}{\partial{\textbf{x}}}-\left(\textbf{w},\,\mbox{rot}[\textbf{a}\times
\textbf{w}]\right)\textbf{a}\vspace{0.20cm}\\=\displaystyle\frac{\partial
\dfrac{1}{2}||\textbf{v}||^2}{\partial{\textbf{x}}}-\Omega_2\wedge\Omega_3(\textbf{v},\mbox{rot}\textbf{v})
\textbf{a},
\end{array}
\end{equation}
where $\displaystyle{\frac{\partial
}{\partial\dot{\textbf{x}}}}=\left(\partial_{\dot{x}^1},\ldots ,\partial_{\dot{x}^N}\right)^T,\quad \displaystyle\frac{\partial
}{\partial{\textbf{x}}}=\left(\partial_{{x}^1},\ldots ,\partial_{{x}^N}\right)^T.$
\end{corollary}
\begin{corollary}\label{main2}
Let us introduce the notation
\[\left|
\begin{array}{ccccc}
 df_1(\partial_1)& \ldots& df_1 (\partial_N) \\
  \vdots & & \vdots \\
   df_{N-1}(\partial_1)
& \ldots & df_{N-1}(\partial_N)\\
\partial_1 &  \ldots & \partial_N \end{array}
         \right|=\{f_1,\,f_2,\,\ldots,f_{N-1},\,*\}
\]
and suppose that  in Theorem \ref{main} the given independent
1-form   are such that $\Omega_j=df_j(x)$ for $ j=1,2,...,N-1,$
and the arbitrary 1-form it is also exact, i.e.
$\Omega_N=df_N,\quad \Omega_N (\textbf{v})=\lambda_N,$ and such
that $\Upsilon=\{f_1,\,f_2,\,\ldots,f_{N-1},\,f_N\}\ne{0}.$

Then the equations \eqref{L0} and \eqref{L2} take the form
respectively
\begin{equation}\label{C0}
\begin{array}{ll}
\dot{\textbf{x}} = -\lambda_N\dfrac{\{f_1,\,f_2,\,\ldots,f_{N-1},\,\,*\}}{\{f_1,\,f_2,\,\ldots,f_{N-1},f_N\}}=
\lambda\{f_1,\,f_2,\,\ldots,f_{N-1},\,\,*\}, \vspace{0.20cm}\\
\displaystyle\frac{d}{dt}\displaystyle\frac{\partial
T}{\partial\dot{x}^k}-\displaystyle\frac{\partial
T}{\partial{x}^k} =\displaystyle\frac{\partial
\dfrac{1}{2}||\textbf{v}||^2}{\partial{x}^k}+\lambda\sum_{j=1}^{N-1}a_{Nj}df_j(\partial_k),
\end{array}
\end{equation}
 where $\lambda_N=\lambda_N(x)$ is an arbitrary function and $a_{Nj}=a_{Nj}(x)$ for $j=1,\ldots N-1$ are element of the skew symmetric matrix $A.$
\end{corollary}
\begin{corollary}\label{main0}
Differential equations \eqref{L2} are Lagrangian  with Lagrangian function
$ L=\dfrac{1}{2}||\dot{\textbf{x}}-\textbf{v}(x)||^2,$ where $ \textbf{v} $ is Cartesian vector field.
\end{corollary}

The proofs of Theorem \ref{main} , Corollary \ref{main1},\,
Corollary \ref{main2},\, and Corollary \ref{main0} are given in
section \ref{s2}.

\section{Proof of the main results}\label{s2}
 \begin{proof}[Proof of Theorem \ref{main}]
  Firstly we shall introduce the following notation and concepts.
Let  $\xi\in (\textsc{Q})$ be  the Lie algebra of vector fields on ${\textsc{Q}}$ and let $\nabla$ be the connection \[ \begin{array}{cc}
\nabla:\xi (\textsc{Q})&\times\xi (\textsc{Q})\longmapsto \xi (\textsc{Q}),\\
&(\textbf{u},\textbf{v})\longmapsto \nabla_\textbf{u
v},\end{array}\] which is $\mathbb{R}$ lineal with respect to
$\textbf{v}$ and $C^{\infty}$ lineal with respect to $\textbf{u}$
and is compatible with metric $G,$ i.e.
$\nabla_\textbf{u}{G}(\textbf{v},\textbf{w})=0,\quad
\forall{{\textbf{u},\textbf{v},\textbf{w}}}\in\xi (\textsc{Q}).$

Let ${\tilde{\textbf{v}}}\in\xi (\textsc{Q})$ be a vector field:
\[
 \tilde{\textbf{v}}=\frac{1}{\Upsilon}\left|
\begin{array}{ccccc}
\Omega_1(\partial_1)&\Omega_1(\partial_2)&\hdots  &
\Omega_1(\partial_N)&\lambda_1\\
\vdots          & \vdots         & \hdots & \vdots&\vdots          \\
\Omega_M(\partial_1)&\Omega_M(\partial_2)&\hdots  &
\Omega_M(\partial_N)
&\lambda_2\\
\Omega_{M+1}(\partial_1)&\Omega_{M+1}(\partial_2)&\hdots
&\Omega_{M+1}(\partial_N)&\lambda_{M+1}\\
\vdots          & \vdots         & \hdots & \vdots &\vdots         \\
\Omega_{N}(\partial_1)
&\Omega_{N}(\partial_N)&\lambda_N\\
\partial_1&\partial_2&\hdots  & \partial_N&0\end{array}
         \right|
\]
where
$\Upsilon\equiv{\Omega_1\wedge{\Omega_2}...\wedge{\Omega_N}(\partial_1,\partial_2,...,
\partial_N)}\ne{0}.$

   The functions $\lambda_j$ for $j=1,\ldots N,$ are arbitrary functions on $\textsc{Q}$ such that $\Omega_j( {\tilde{\textbf{v}}} )=\lambda_j,$ for $j=1,\ldots N.$

 Let $\tilde{\sigma}$ be the 1-form associated with the vector
field $\tilde{\textbf{v}},$ i.e.
$$\tilde{\sigma}=(\tilde{\textbf{v}}(x),dx)=\sum_{j,k=1}^NG_{jk}\tilde{v}^j(x)dx^k\equiv{
\sum_{k=1}^N\tilde{p}_kdx^k},$$ then, in view of the condition
$\Upsilon\ne 0$ the 2-form $d\tilde{\sigma}$ admits the
development
$d\tilde{\sigma}=\dfrac{1}{2}\displaystyle\sum_{j,k=1}^N\tilde{a}_{jk}(x)
\Omega_j\wedge{\Omega_k},$ where $A=(\tilde{a}_{jk})$ is a
$N\times N$ skew-symmetric matrix such that
$$
\tilde{a}_{jk}=(-1)^{j+k-1}\frac{1}{\Upsilon}d\sigma\wedge
\Omega_1\wedge...\wedge\widehat{\Omega}_{k}
....\wedge\widehat{\Omega}_{j}....
\wedge\Omega_{N}(\partial_1,\partial_2,...,\partial_N),
$$
 $\widehat{\Omega}_j,\, \widehat{\Omega}_k,$ means that these elements are
omitted.

\smallskip

In view of the relations
\[
\iota_{\tilde{\textbf{v}}}d\tilde{\sigma}=\dfrac{1}{2}\displaystyle\sum_{k,j=1}^N\left(\tilde{a}_{jk}\Omega_j(\tilde{\textbf{v}})
- \tilde{a}_{kj}\Omega_j(\tilde{\textbf{v}})\right)\Omega_k=\displaystyle\sum_{k,j=1}^N\tilde{a}_{jk}\Omega_j(\tilde{\textbf{v}})\Omega_k,
 =\displaystyle\sum_{k,j=1}^N\tilde{a}_{jk}\lambda_j\Omega_k
\]
we obtain  that the contraction of $d\tilde{\sigma}$ along $\tilde{\textbf{v}}$ is
\begin{equation}\label{31}
\iota_{\tilde{\textbf{v}}}d\tilde{\sigma}=\sum_{j=1}^N \tilde{\Lambda}_j\Omega_j,
\end{equation}
where
${\tilde{\Lambda}}=\left(\tilde{\Lambda}_1,\tilde{\Lambda}_2,...,\tilde{\Lambda}_N\right)^T.$
It is easy to check that this
 vector can be calculated as follows
\begin{equation}\label{2}
{\tilde{\Lambda}}=A\tilde{\lambda}=M^{-1}\tilde{\tau},\,
\end{equation}
where $M=(\Omega_j( \partial_k))$  and
$\tilde{\lambda}=\left(\lambda_1,\ldots,\lambda_N\right)^T,\,
\tilde{\tau}=\left(\iota_{\tilde{\textbf{v}}}d\tilde{\sigma}(\partial_1)\,,\ldots,\iota_{\tilde{\textbf{v}}}d\tilde{\sigma}(\partial_N)\right)^T.$
\smallskip

Now we prove that the differential system
\begin{equation}
\label{32}\dot{\textbf{x}}={\tilde{\textbf{v}}}(x),\quad  x \in\textsc{Q},
\end{equation}
is invariant relationship of the Lagrangian equations with
Lagrangian function
$$
\tilde{L}=\dfrac{1}{2}||\dot{\textbf{x}}-{\tilde{\textbf{v}}}||^2=
\frac{1}{2}\sum_{j,k=1}^NG_{kj}(\dot{x}^j-\tilde{v}^j)(\dot{x}^k-\tilde{v}^k).
$$

 Indeed after covariant differentiation we obtain
 $\nabla_{\dot{\textbf{x}}}(\dot{\textbf{x}}-{\tilde{\textbf{v}}}) =0,$
 or, what is the same
 $$
 \nabla_{\dot{\textbf{x}}}G\left(\dot{\textbf{x}}-\tilde{\textbf{v}}\right) =\nabla_{\dot{\textbf{x}}}\left(\displaystyle\frac{\partial
\tilde{L}}{\partial\dot{\textbf{x}}}\right)=0,
 $$
 hence , by considering that
 \[ \begin{array}{ll}
 \displaystyle\nabla_{\dot{\textbf{x}}}\left(G\dot{\textbf{x}}\right)=&\displaystyle\frac{d}{dt}\displaystyle\frac{\partial
T}{\partial\dot{\textbf{x}}}-\displaystyle\frac{\partial
T}{\partial{\textbf{x}}},\quad T=\frac{1}{2}||\dot{\textbf{x}}||^2 \vspace{0.20cm}\\
  \displaystyle\nabla_{\dot{\textbf{x}}}\tilde{p}_j=& \displaystyle{\sum_{j=1}^N}\dot{x}^j\left(\nabla_{\partial_j}\tilde{p}_k-\nabla_{\partial_k}\tilde{p}_j\right)+\displaystyle{\sum_{j=1}^N}\dot{x}^j\nabla_{\partial_k}\tilde{p}_j\vspace{0.20cm}\\
 =& \displaystyle{\sum_{j=1}^N}\dot{x}^j\left({\partial_j}\tilde{p}_k-{\partial_k}\tilde{p}_j\right)+\displaystyle{\sum_{j=1}^N}v^j\nabla_{\partial_k}\tilde{p}_j+
\displaystyle{\sum_{j=1}^N}\left(\dot{x}^j-\tilde{v}^j\right)\nabla_{\partial_k}\tilde{p}_j\vspace{0.20cm}\\
 =&\displaystyle\frac{d}{dt}\displaystyle\frac{\partial
V}{\partial\dot{x}^k}-\displaystyle\frac{\partial
V}{\partial{x}^k} +\sum_{j=1}^N\left(\dot{x}^j-\tilde{v}^j\right)\nabla_{\partial_k}\tilde{p}_j
\end{array}\]

where $\tilde{p}_k=\displaystyle{\sum_{j=1}^N}G_{kj}{\tilde{v}}^j,\,V=(\dot{\textbf{x}},\tilde{\textbf{v}})-\frac{1}{2}||\tilde{\textbf{v}}||^2,$
and $(\,\,,\,\,)$ is the scalar product. Then along the solutions of \eqref{32} we  give
 $
  \displaystyle\frac{d}{dt}\displaystyle\frac{\partial
\tilde{L}}{\partial\dot{x}^k}-\displaystyle\frac{\partial
\tilde{L}}{\partial{x}^k}=0,
$
 for $k=1,2,\ldots N.$ It is easy to show that these Lagrangian equations admits the
representation
\[
\displaystyle\frac{d}{dt}\displaystyle\frac{\partial
T}{\partial\dot{x}^k}-\displaystyle\frac{\partial
T}{\partial{x}^k} = \tilde{\omega} (\partial_k)+
\nabla_{\dot{\textbf{x}}-\tilde{\textbf{v}}}\tilde{p}_k,
\]
where
$\tilde{\omega}=\frac{1}{2}d||{\tilde{\textbf{v}}}||^2+\iota_{\tilde{\textbf{v}}}d\tilde{\sigma}.$
  In view of \eqref{31} and \eqref{32} we finally deduce the differential equations
  \begin{equation}\label{3}
\displaystyle\frac{d}{dt}\displaystyle\frac{\partial
T}{\partial\dot{x}^k}-\displaystyle\frac{\partial
T}{\partial{x}^k} =\displaystyle\frac{\partial
\dfrac{1}{2}||\tilde{\textbf{v}}||^2}{\partial{x}^k} +\sum_{j=1}^N
\tilde{\Lambda}_j\Omega_j(\partial_k).
\end{equation}

  If we determine the vector field $ \textbf{v}$ and functions $ \Lambda_1,\,\Lambda_2,\ldots \Lambda_N$ as follows
  \begin{equation}\label{S11}
   \textbf{v}= \tilde{\textbf{v}}|_{ \lambda_1=\lambda_2=\ldots=\lambda_M=0},\quad \Lambda_j=\tilde{\Lambda}_j|_{ \lambda_1=\lambda_2=\ldots=\lambda_M=0},\quad j=1,2,\ldots N\end{equation}
   and require that $\Lambda_j=0,\quad j=M+1,...,N,$
  then we obtain the Cartesian vector field and differential system \eqref{3} coincide with system \eqref{L2}.
In short Theorem \ref{main} is proved.\end{proof}
 \begin{proof}[Proof of Corollary \ref{main1}]
The vector field \eqref{00} in this case takes the form
\[
\begin{array}{cc}
\Upsilon \textbf{v}&= \left|
\begin{array}{ccccc}
\Omega_1(\partial_x)&\Omega_1(\partial_y)&\Omega_1(\partial_z)&0\\
   \Omega_2(\partial_x)&\Omega_2(\partial_y)&\Omega_2(\partial_z)&\lambda_2    \\
\Omega_3(\partial_x)&\Omega_3(\partial_y)&\Omega_3(\partial_z)&\lambda_3\\
\partial_x&\partial_y& \partial_z&0\end{array}
         \right|\vspace{0.20cm}\\
&=\lambda_2\left|
\begin{array}{ccccc}
\Omega_1(\partial_x)&\Omega_1(\partial_y)&\Omega_1(\partial_z)\\
   \Omega_3(\partial_x)&\Omega_3(\partial_y)&\Omega_3(\partial_z)\\
\partial_x&\partial_y& \partial_z\end{array}
         \right|\vspace{0.20cm}\\
&-\lambda_3\left|
\begin{array}{ccccc}
\Omega_1(\partial_x)&\Omega_1(\partial_y)&\Omega_1(\partial_z)\\
   \Omega_2(\partial_x)&\Omega_2(\partial_y)&\Omega_2(\partial_z)\\
\partial_x&\partial_y& \partial_z\end{array}
         \right|\vspace{0.20cm}\\
&=\left|
\begin{array}{ccccc}
a_1&a_2&a_3\\
   w_1&w_2&w_3   \\
\partial_x&\partial_y& \partial_z\end{array}
         \right|,
\end{array}
\]
thus $\textbf{v}(x)=[\textbf{a}\times \textbf{w}].$

\smallskip

 On the other hand  considering that
\[ d\sigma=(\mbox{rot}\textbf{v})_xdz\wedge dy+(\mbox{rot}\textbf{v})_ydx\wedge dz+(\mbox{rot}\textbf{v})_zdy\wedge dx\]
where
$(\mbox{rot}\textbf{v})_x=dx(\mbox{rot}\textbf{v}),\,\,(\mbox{rot}\textbf{v})_y=dy(\mbox{rot}\textbf{v}),\,\,(\mbox{rot}\textbf{v})_z=
dz(\mbox{rot}\textbf{v}),$ we get
\[ \imath_{ \textbf{v} }d\sigma=\left([\textbf{v}\times \mbox{rot}\textbf{v}],d\textbf{x}\right)=\Lambda_1\Omega_1+\Lambda_2\Omega_2+\Lambda_3\Omega_3\]
which lead to
\[
 \begin{array}{cc}
\Upsilon\Lambda_1&=\Omega_2\wedge\Omega_3(\textbf{v},\,\mbox{rot}\textbf{v})=\lambda_2\Omega_3(\mbox{rot}\textbf{v})-\lambda_3\Omega_2(\mbox{rot}\textbf{v})=
\vspace{0.20cm}\\
&=\left(\textbf{w},\,\mbox{rot}\textbf{v}\right),\vspace{0.20cm}\\
 \Upsilon\Lambda_2&=\Omega_3\wedge\Omega_1(\textbf{v},\,\mbox{rot}\textbf{v})=\lambda_3\Omega_1(\mbox{rot}\textbf{v})-\lambda_1\Omega_3(\mbox{rot}\textbf{v}),
\vspace{0.20cm}\\
&=\lambda_3\Omega_1(\mbox{rot}\textbf{v})=\lambda_3\left(\textbf{a},\mbox{rot}\textbf{v}\right),
\vspace{0.20cm}\\
\Upsilon\Lambda_3&=-\Omega_2\wedge\Omega_1(\textbf{v},\,\mbox{rot}\textbf{v})=-\lambda_2\Omega_1(\mbox{rot}\textbf{v})+\lambda_1\Omega_2(\mbox{rot}\textbf{v}),
\vspace{0.20cm}\\
&=-\lambda_2\Omega_1(\mbox{rot}\textbf{v})=-\lambda_2\left(\textbf{a},\mbox{rot}\textbf{v}\right),
\end{array}
\]
here we put $ \Omega_1(\textbf{v})=\lambda_1=0.$

From the conditions $ \Lambda_2=\Lambda_3=0$ we obtain
\[\begin{array}{ccc}
  \left(\textbf{a},\mbox{rot}\textbf{v}\right)=
  \left(\textbf{a},\mbox{rot}[\textbf{a}\times \textbf{w}]\right)=0,\vspace{0.20cm}\\
\end{array}
\]
hence we easily deduce differential system  \eqref{L1} and
\eqref{LLL1}.

 \smallskip

In short Corollary \ref{main1} is proved.
\end{proof}

\begin{proof}[Proof of Corollary \ref{main2}]
In this case we obtain that the vector field $\textbf{v}$ takes the form
\[
\begin{array}{lll}
 \textbf{v}=\displaystyle{\frac{1}{\Upsilon}}\left|
 \begin{array}{ccccc}
df_1(\partial_1)&\hdots  &
df_1(\partial_N)&0\\
\vdots                & \hdots & \vdots&\vdots          \\
%\vdots          & \vdots         & \hdots & \vdots &\vdots         \\
df_{N-1}(\partial_1)&\hdots
&df_{N-1}(\partial_N)&0\\
df_{N}(\partial_1)&\hdots
&df_{N}(\partial_N)&\lambda_N\\
\partial_1&\hdots  & \partial_N&0\end{array}
         \right|=-\displaystyle{\frac{\lambda_N}{\Upsilon}} \left|
\begin{array}{ccccc}
 df_1(\partial_1)& \ldots& df_1 (\partial_N) \\
  \vdots & & \vdots \\
   df_{N-1}(\partial_1)
& \ldots & df_{N-1}(\partial_N)\\
\partial_1 &  \ldots & \partial_N \end{array}
         \right|\vspace{0.20cm}\\
=-\lambda_N\dfrac{\{f_1,\,f_2\ldots,f_{N-1},\,*\}}{\{f_1,\,f_2\ldots,f_{N-1},\,f_N\}}=\lambda\{f_1,\,f_2\ldots,f_{N-1},\,*\}.
         \end{array}
\]

On the other hand by considering that $\lambda_j=0$ for $j=1,2,\ldots N-1,$ from \eqref{02} follows that
$
\Lambda_j=a_{Nj}\lambda_N,\quad j=1,\ldots, N-1,\quad \Lambda_N=a_{NN}\lambda_N=0.
$

Clearly the last equation is satisfied identically in view that $(a_{jk})$ are elements of the skew-symmetric matrix $A.$

\smallskip

Therefore we easily deduced the differential equations generated from Cartesian and Lagrangian approach.

Thus Corollary \ref{main2} has been proven.
\end{proof}

 \begin{proof}[Proof of Corollary \ref{main0}]
Follows from the proof of Theorem \ref{main} by considering
\eqref{S11}
\end{proof}
\smallskip

 \section{ Decartes approach for non-holonomic system with three degree of freedom  and
one constraints .}

In this section we apply the corollary \ref{main1} to study the Chapliguin-Catatheodory sleight and  Suslov's problem for the rigid body around a fixed point.

{\it Chapliguin-Carathodory's sleigh}

\smallskip
We shall now analyze one of the classical nonholonomic
systems  Chapliguin-Carathodory's sleigh (which we call a sleigh). \cite{NF}

\smallskip

 The idealized sleigh is a body that has three points of contact with the plane.
Two of them slide freely but the third, $A,$ behaves like a knife
edge subjected to a constraining force $ \textbf{R}$ which does not
allow transversal velocity.  More precisely, let $yoz$ be an
inertial frame and $\xi\,A\eta$ a frame moving with the sleigh.
Take as generalized coordinates the Decartes coordinates of the
center of mass $C$ of the sleigh and the angle $x$ between the $y$
and the $\xi$ axis. The reaction force $\textbf{R}$  against the
runners is exerted laterally at the point of application $A$ in
such a way that the $\eta$ component of the
 velocity is zero. Hence, one has the constrained system
 with the configuration space
$\textsc{Q}=S^1\times\mathbb{R}^2,$ with the kinetic energy
$T=\frac{m}{2}(\dot{y}^2+\dot{z}^2)+\frac{I_c}{2}\dot{x}^2,$ and
with the constraint
$\epsilon\dot{x}+\sin{x}\dot{y}-\cos{x}\dot{z}=0,$ where $m$ is
the mass of the system and $I_c$ is the moment of inertia about a
vertical axis through $C$ and $\epsilon=|AC|.$

\smallskip

Observe that the "javelin" (or arrow or Chapliguin's skate) is a particular case of
a sleigh and can be obtained when $\epsilon=0.$

\smallskip

To apply the Decartes approach for this system, first we
introduce the 1-form $\Omega_j,$ for $j=1,2,3$ in such a
way that the determinant $\Upsilon\ne{0}.$
 This condition holds in particular if
$$\Omega_1=\epsilon d{x}+\sin{x}d{y}-\cos{x}d{z},\quad
\Omega_2= \sin{x}d{z}+\cos{x}d{y},\quad
        \Omega_3=dx,$$
so  $\Upsilon=\Omega_1\wedge\Omega_2\wedge\Omega_3(\partial_x,\,\partial_y,\,\partial_z)=1.$

\smallskip

The Descartes  approach  produce the differential equations (see formula \eqref{L1})
\cite{Sad}

\smallskip

\begin{equation}\label{ch1}
\dot{x}=\lambda_3,\quad\dot{y}=\lambda_2\cos{x}-\epsilon \lambda_3\sin{x},\quad\dot{z}=\lambda_2\sin{x}+\epsilon\lambda_3\cos{x},
\end{equation}
 here $\lambda_j=\lambda_3(x,y,z,\epsilon )$ for $j=2,3$ are solutions
of the  equation
\begin{equation}\label{cch2}
  \sin x(J\partial_z\lambda_3+\epsilon
m\partial_y\lambda_2)+\cos x(J\partial_y\lambda_3-\epsilon
m\partial_z\lambda_2)-m(\partial_x\lambda_2-\epsilon\lambda_3)=0,
\end{equation}
where $J=J_C+\epsilon^2m,\quad
||\textbf{v}||^2=J\lambda^2_3+m\lambda^2_2.$

\smallskip

Now we show that there are solutions of \eqref{ch1} and
\eqref{cch2} fully describes the inertial movements of the sleigh.

\smallskip

\begin{corollary}
 All the inertial trajectories of  Chapliguin -Carathodory's sleigh can be obtained from Cartesian approach.
\end{corollary}
\begin{proof}

Let us suppose that $\lambda_j=\lambda_j(x,\epsilon)$ for $j=2,3.$
Clearly that in this case \eqref{cch2} takes the form
$\partial_x\lambda_2-\epsilon\lambda_3=0$ and  all paths of the
equation \eqref{ch1} can be obtained from the formula
\[
\begin{array}{ll}
y=y_0+\displaystyle\int\dfrac{(\lambda_2(x,\epsilon )\cos{x}-
\epsilon \lambda_3\sin{x})dx}{\lambda_3(x,\epsilon )},\vspace{0.20cm}\\
{z}=z_0-\displaystyle\int\dfrac{(\lambda_2(x,y,z,\epsilon
)\sin{x}-\epsilon\lambda_3\cos{x})dx}{\lambda_3(x,\epsilon )},\vspace{0.20cm}\\
t=t_0+\displaystyle\int\dfrac{dx}{\lambda_3(x,\epsilon )}.
\end{array}%
\]

 On the other hand, the inertial motions of the sleigh can be obtained from the equations deduce from the classical approach
\[
J_C\ddot{x}=\epsilon\mu ,\quad m\ddot{y}=\mu\sin x,\quad m\ddot{z}=-\mu\cos x,\quad \epsilon\dot{x}+\sin x\dot{y}-\cos x\dot{z}=0.
\]
 Hence, after straightforward calculations we get the first order ordinary differential equations
\[
\dot{x}=qC\cos\theta,\,\dot{y}=C(\sin\theta\cos
x-q\epsilon\cos\theta\sin x),\,\dot{z}=C(\sin\theta\sin
x+q\epsilon\cos\theta\cos x),
\] where $ \theta=q\epsilon x+C,\quad q=\sqrt{\displaystyle\frac{m}{J}}$

\smallskip

 This system can be obtained from  \eqref{ch1} if we choose
$$\lambda_2=C\sin\theta,\quad \lambda_3=Cq\cos\theta.$$
Is clear that in this case
$$2||\textbf{v}||^2=J\lambda^2_3(x,\epsilon\,)+m\lambda^2_2(x,\epsilon\,)={mC^2},$$ therefore the sleigh moves by inertia. So the corollary is proved.
\end{proof}

\smallskip

Now we study  Chapliguin's skate.  Cartesian approach in this case produce the differential equations, which can be obtained from \eqref{ch1}
and \eqref{cch2} by putting $ \epsilon =0.$
\begin{equation} \label{ch22}
\begin{array}{ll}
\dot{x}=&\lambda_3(x,y,z,0 ),\quad \dot{y}=\lambda_2(x,y,z,0 )\cos{x},\quad\dot{z}=\lambda_2(x,y,z,0 )\sin{x},\vspace{0.20cm}\\
 &J(\sin
x\partial_z\lambda_3+\cos
x\partial_y\lambda_3)-m\partial_x\lambda_2=0
\end{array}%
\end{equation}
\begin{corollary}
All the trajectories of Chapliguin's skate with the initial
condition $ \dot{x}(t_0)=C_0\ne 0$ and under the action of the
potential field of force with potential function $U=mgy$ can be
obtained from the Cartesian approach.
\end{corollary}
\begin{proof}
In fact, for the case when $\epsilon=0$ the classical approach for
Chapliguin-Carathodory's sleigh gives (Chapliguin's skate) the following equations of
motion
\[
\ddot{x}=0,\quad\ddot{y}=g+\mu\sin x ,\quad\ddot{z}=-\mu\cos x,\quad\sin x\dot{y}-\cos x\dot{z}=0
\]
 Hence, integrating we deduce the following differential system of first order (see for instance \cite{Sad})
\[
\begin{array}{ll}
\dot{x}=C_0\ne 0,\quad \dot{y}=(\displaystyle\frac{g\sin x}{C_0}+C_1)\cos x,\quad\dot{z}=(\displaystyle\frac{g\sin x}{C_0}+C_1)\sin x.
\end{array}%
\]
Let $\textbf{b}$  be the vector field associated with this differential system, i.e. $$ \textbf{b}= \left(C_0,\,(\displaystyle\frac{g\sin x}{C_0}+C_1)\cos x,\,(\displaystyle\frac{g\sin x}{C_0}+C_1)\sin x\right).$$

\smallskip

Whereas   $(\textbf{a},\mbox{rot}\textbf{b})=\dfrac{g\cos
x}{C_0}\ne{0}.$

\smallskip

Denoting by $ \kappa =\displaystyle\frac{C_0}{{g\sin
x}{C_0}+C_1)},$ we easily obtain that $(\textbf{a},\mbox{rot}(\kappa\textbf{
b}))=0$,  so  $\textbf{b}$ is Cartesian equivalent vector
field.

\smallskip

The vector field $ \kappa \textbf{b}$ can be
obtained from \eqref{ch22} by choosing $\lambda_3=\kappa
$ and $\lambda_2=C_0.$

Summarizing the corollary is proved.
\end{proof}

\smallskip

{\it The rigid body around a fixed point in the Suslov case}

 \smallskip

in this section we study one classical problem of non-holonomic
dynamics formulated by Suslov \cite{Su,Koz3}. We consider the
rotational motion of a rigid body around a fixed point and subject
to the non-holonomic constraints $(\tilde{\textbf{a}},\omega)=0$
where ${\omega}=(\omega_1,\,\omega_2,\,\omega_3)$ is a body angular velocity and
$\tilde{\textbf{a}}$ is a constant vector. Suppose the body
rotates in an force field with potential
$U(\gamma_1,\gamma_2,\gamma_3)$. Applying the method of Lagrange
multipliers we write the equations of motion in the form

\begin{equation}
\label{81}
\begin{array}{ll}
  I\dot{{\omega}}= &[I\omega\times{\omega}]+[{\gamma}\times\frac{\partial
U}{\partial{\gamma}}]+\mu \tilde{\textbf{a}}, \quad
 \dot{{\gamma}}=[{\gamma}\times{\omega}]\vspace{0.20cm} \\
  &(\tilde{\textbf{a}},{\omega})=0
\end{array}%
\end{equation}
Where $ \gamma=\gamma_1,\gamma_2,\gamma_3)=(\sin z\sin x,\,\sin z\cos
x,\,\gamma_3=\cos z),
$,\,  $I$  is the
inertial tensor of the body, $\mu$ is the Lagrange multiplier which can be expressed
as a function of $\omega$ and $\gamma$ as follows
\[\mu=-\dfrac{\left(\textbf{a},[I\omega,\,{\omega}]+[{\gamma},\,\frac{\partial
U}{\partial{\gamma}}]\right)}{(\textbf{a},I^{-1}\textbf{a})}.\]

\smallskip

  It is well-known the following result \cite{Koz3}
\begin{proposition}
If $ \textbf{a}$ is an eigenvector of operator $I,$ i.e.
 \begin{equation}\label{K}
 I\textbf{a}=\kappa\textbf{a},
 \end{equation}
 then the phase flow of system \eqref{81} preserves the "standard"measure in $ \mathbb{R}^6=\mathbb{R}^3\{\omega\}\times\mathbb{R}^3\{ \gamma\}.$
\end{proposition}
G.K.Suslov has considered a particular case when the body is not
under action of exterior forces: $U\equiv 0$. If \eqref{K} holds
then the equations  \eqref{81} have the additional first integral
$K_4=(I\omega ,I\omega ).$  E.I.Kharlamova in \cite{Kharlamov}
study the case when the body rotates in the homogenous force field
with the potential  $U=( \textbf{b},\gamma )$ where $ \textbf{b}$
is an orthogonal to $ \textbf{a}$ vector. Under these conditions
the equation of motion have the first integral $ K_4=( I\omega
,\textbf{b}).$ V.V. Kozlov in \cite{Koz1} consider an opposite
case, when $ \textbf{b}=\epsilon \textbf{a},\,\epsilon\ne{0}.$ The
integrability problem in this case was study in particular in
\cite{Koz3,MP}. The case when  $U=\epsilon \det{I}(I^{-1}\gamma
,\,\gamma )$ the system \eqref{81} have the Clebsch-Tisseran first
integral $K_4=\frac{1}{2}(I\omega ,I\omega ) -\dfrac{1}{2}\epsilon
\det{I}(I^{-1}\gamma ,\,\gamma ),$ \cite{Koz3}.

{From} now on, we suppose that equality \eqref{K} is fulfilled. We
assume that vector $\textbf{a}$ coincides with one of the
principal axes and without loss of generality we can choose it as
the third axis, i.e., $\textbf{a}=(0,0,1)$ (see for more details
\cite{Koz3}) Equations of motion have the following form

\begin{equation}
\label{222}
\left\{%
\begin{array}{ll}
I_1\dot{\omega}_1=\gamma_3\partial_{\gamma_2}U-\gamma_2\partial_{\gamma_3}U,\quad
I_2\dot{\omega}_2=\gamma_1\partial_{\gamma_3}U-\gamma_3\partial_{\gamma_1}U,\vspace{0.2cm}\\
(I_1-I_2)\omega_1\omega_2+\gamma_2\partial_{\gamma_1}U-\gamma_1\partial_{\gamma_2}U+\mu=0,\quad  \vspace{0.2cm}\\
\dot{\gamma}_1=-\gamma_3\omega_2,\quad
\dot{\gamma}_2=\gamma_3\omega_1,\quad
\dot{\gamma}_3=\gamma_1\omega_2-\gamma_2\omega_1,
\end{array}%
\right.
\end{equation}
where $I_1,\,I_2$ are the principal moments of inertia of the body with respect to the
 first and the second axis.
We observe that the above mentioned choice of $ \textbf{a}$
guarantees that the phase flow of system \eqref{222} preserves the
standard measure in $ \mathbb{R}^5 \{
\omega_1,\,\omega_2,{\gamma}\}$.

In \cite{Ram6} we prove the following result

\begin{theorem}\label{P1}

 Let us suppose that the body rotates within the force field defined by  the potential
\begin{equation}\label{333}
  U=\frac{1}{2I_1I_2}(I_1\mu^2_1+I_2\mu^2_2)-h,
\end{equation}
 where $h$ is a constant and $\mu_1,\,\mu_2$ are the solutions of the following  first order partial
differential equation
\begin{equation}\label{4}
\gamma_3(\frac{\partial \mu_1}{\partial\gamma_2}- \frac{\partial
\mu_2}{\partial \gamma_1})-\gamma_2\frac{\partial
\mu_1}{\partial\gamma_3}+\gamma_1\frac{\partial \mu_2}{\partial
\gamma_3}=0,
\end{equation}
 then the following statements hold:
 \begin{itemize}
 \item[(I)]

 The equations \eqref{222} have the  first integrals
  \begin{equation}\label{5}
 I_1\omega_1-\mu_2=0,\quad
I_2\omega_2+\mu_1=0,
\end{equation}
consequently, they are integrable  by quadratures. In particular
\begin{equation}\label{s}
\begin{array}{rl}
I_1\omega_1=&\displaystyle\frac{\partial
S(\gamma_1,\gamma_2,K_2)}{\partial\gamma_2}+
\Psi_2(\gamma^2_1+\gamma^2_3,K_2,\gamma_2)\vspace{0.2cm}\\&+\gamma_1\Omega
(\gamma^2_1+\gamma^2_2,K_2,\gamma_3),\vspace{0.2cm}\\
I_2\omega_2=&-\displaystyle\frac{\partial
S(\gamma_1,\gamma_2,K_2)}{\partial\gamma_1}- \Psi_1(\gamma^2_1
+\gamma^2_3,K_2,\gamma_1)\vspace{0.2cm}\\
&-\gamma_2\Omega (\gamma^2_1+\gamma^2_2,K_2,\gamma_3),
\end{array}%
\end{equation}
are constants on the solutions of \eqref{222}, where
$K_2=\gamma^2_1+\gamma^2_2+\gamma^2_3$ and
$S,\Psi_1,\,\Psi_2,\,\Omega$ are arbitrary smooth functions .

\item[(II)] The Suslov's,\,Kharlamova-Zabelina's,\,Kozlov's ,
Clebsch-Tisseran's and Tisseran-Okunova's first integrals can be obtained from \eqref{s}.

\item[(III)] The dependence $ \gamma = \gamma (t)$ we determine by quadratures of the Poisson equations
 which in this case take the form:
\begin{equation}\label{P}
\dot{\gamma}_1=\dfrac{\gamma_3\mu_1}{I_2},\quad
 \dot{\gamma}_2=\dfrac{\gamma_3\mu_2}{I_1},\quad
 \dot{\gamma}_3=-\dfrac{\gamma_1\mu_1}{I_2}-\dfrac{\gamma_2\mu_2}{I_1}.
   \end{equation}
\end{itemize}

\end{theorem}

It is interesting to note that the proof of the Theorem \ref{P1}
first was obtained using Cartesian  approach (for more details see
\cite{Ram4}) which we proposed below .

\smallskip

\smallskip

Let us suppose that manifold $\textsc{Q}$ is the special
orthogonal group of rotations of $\mathbb{E}^3,$ i.e.
$\textsc{Q}=SO(3),$ with the Riemann metric
\[G=\left(
\begin{array}{ccc}
  I_3 & I_3\cos{z} & 0 \\
  I_3\cos{z} & (I_1\sin^2 x+I_2\cos^2 x)\sin^2{z}+I_3\cos^2{z} & (I_1-I_2)\sin x\cos x\sin z \\
  0 & (I_1-I_2)\sin x\cos x\sin z & I_1\cos^2x+I_2\sin^2x\\
\end{array}\right)
\]

$\det G=I_1I_2I_3\sin^2z.$

\smallskip

In this case we have that the constraint is
$\omega_3=\dot{x}+\cos\,z \dot{y}=0.$

Hence $\textbf{a}=(1,\,\cos z,\,0).$  By choosing the 1-form $\Omega_j$ for $j=1,2,3$ as
follow $$\Omega_1=dx+\cos z dy,\quad \Omega_2=dy,\quad \Omega_3=dz.$$ Consequently
$\Upsilon=d\Omega_1\wedge d\Omega_2\wedge d\Omega_3(\partial_x,\,\partial_y,\,\partial_z)=1,$ and
$$\textbf{v}=\lambda_2(\cos\,z\,\partial_x-\partial_y)-\lambda_3\partial_z$$
Thus we obtain that
\[
\begin{array}{lll}
&p_1=0,\vspace{0.20cm}\\& p_2=(I_3-I_1+(I_1-I_2)\cos^2x)\cos
z\sin^2z\lambda_2+(I_1-I_2)\cos x\sin x\sin  z\lambda_3,\vspace{0.20cm}\\
&p_3=(I_2\sin^2x+I_1\cos^2x)\lambda_3+(I_2-I_1)\sin x\cos x\sin
z\cos z\lambda_2\end{array}
\]
The differential equations \eqref{L1} in this cases take the form
respectively

\begin{equation}
\label{86}
\dot{x}=\cos z\,\lambda_2,\quad\dot{y}=-\lambda_2,\quad\dot{z}=-\lambda_3
  \end{equation}
and
\begin{equation}\label{87}
  (\textbf{a},\mbox{rot}\textbf{v})=\partial_zp_2-\partial_yp_3+\cos
z\partial_xp_3=0
\end{equation}
After the change
 $
 \gamma_1=\sin z\sin x,\quad\gamma_2=\sin z\cos
x,\quad\gamma_3=\cos z
$ the system \eqref{86} and condition
\eqref{87} can be written as follow
\begin{equation}
\label{888}
\dot{\gamma_1}=\frac{1}{I_2}\mu_1\gamma_3,\quad\dot{\gamma_2}=\frac{1}{I_1}\mu_2\gamma_3,\quad\dot{\gamma_3}=-\frac{1}{I_1I_2}\left(I_1\mu_1\gamma_1+
I_2\mu_2\gamma_2\right)
\end{equation}

\begin{equation}\label{89}
\sin z(\gamma_3(\frac{\partial \mu_1}{\partial\gamma_2}-
\frac{\partial \mu_2}{\partial \gamma_1})-\gamma_2\frac{\partial
\mu_1}{\partial\gamma_3}+\gamma_1\frac{\partial \mu_2}{\partial
\gamma_3})-\cos x\partial_y\mu_2-\sin
x\partial_y\mu_1=0
\end{equation}
where
$
\mu_2=-I_1(\cos x \lambda_3+\sin x \lambda_2),\quad\mu_1=I_2(-\sin x \lambda_3+\cos x \lambda_2).
 $
 \smallskip

 Clearly if $ \mu_j=\mu_j(x,z),$ for $j=1,2,$  then the equation \eqref{89} coincide with equation \eqref{4}  and \eqref{888} coincide with \eqref{P}.

 \smallskip

\section{Cartesian vector
 field in three dimensional Euclidean space }

Let $\mathbb{E}^3$ be the three dimensional Euclidian space with
Cartesian coordinates $\textbf{x}=(x_1,x_2,x_3).$

We consider a particle with Lagrangian function
$$
L=\frac{1}{2}||\dot{\textbf{x}}||^2-U(x),$$ and constraint
 $ (\textbf{a},\dot{\textbf{x}})=0,$ where $(\,,\,)$ denotes the
 scalar product in $\mathbb{E}^3,$
$\dot{{\textbf{x}}}=(\dot{x}_1,\,\dot{x}_2,\,\dot{x}_3)$ and
$\textbf{a}={\textbf{a}}(x)=(a_1(x)\,,a_2(x),\,a_3(x))$ is a
smooth vector field in
 $\mathbb{E}^3.$ Below we shall use the following notation
 $\partial_\textbf{x}f=(\partial_{x_1}f,\partial_{x_2}f,\partial_{x_3}f)^T.$
\smallskip

The equations of motion in particular for constrained particle in
$\mathbb{R}^3,$ can be deduced from the d$^{'}$Alembert-Lagrange
Principle and are such that
\[
\ddot{\textbf{x}}=\partial_{\textbf{x}}U+\mu\textbf{a},\quad(\textbf{a},\dot{\textbf{x}})=0,
\]
where $\mu$ is the Lagrangian multiplier.
\smallskip

 Cartesian and Lagrangian approach produces the following differential equations
respectively (see formula \eqref{L1} and \eqref{LLL1})

\begin{equation}\label{E1}
\begin{array}{ll}
\dot{\textbf{x}}=&[\textbf{a}\times \textbf{w}],\qquad (\textbf{a},\mbox{rot}[\textbf{a}\times \textbf{w}])=0,\vspace{0.20cm}\\
\ddot{{\textbf{x}}}=&\partial_{\textbf{x}}\left(\dfrac{1}{2}||[\textbf{a}\times
\textbf{w}]||^2\right)+(\mbox{rot}[\textbf{a}\times
\textbf{w}],\textbf{w}) {\textbf{a}},\quad
\end{array}
 \end{equation}
{\it Example}

Suppose that\quad
$\textbf{a}=f_{\textbf{x}},\quad\textbf{w}=\dfrac{\textbf{c}}{c^2},$

  where $f=r+(\textbf{x},\textbf{b})=c^2$
   and $\textbf{b}$ and $\textbf{c}$ are constants vector
field  such that $$(\textbf{x},\textbf{c})=0,\quad
(\textbf{b},\textbf{c})=0,\quad
||\textbf{c}||^2=c^2,\,r=||\textbf{x}||.$$ Then Cartesian and
Lagrangian approach generate the following differential equations
respectively
 $$\dot{\textbf{x}}=\dfrac{[f_{\textbf{x}}\times\textbf{c}]}{c^2},\qquad
 \ddot{\textbf{x}}=-\displaystyle{\dfrac{ \textbf{x}}{r^3}}.$$
Indeed in view of the
relation\,\,$\mbox{rot}[f_{\textbf{x}}\times\textbf{c}]=\dfrac{\textbf{c}}{r},$
we get $$\left(\textbf{a},\mbox{rot}[\textbf{a}\times
\textbf{w}]\right)=\left(
{f_{\textbf{x}}},\dfrac{\textbf{c}}{c^2}\right)=0$$ thus the given
vector field is Cartesian.

\smallskip

{From} the relations  \quad $ (\mbox{rot}[\textbf{a}\times
\textbf{w}],\textbf{w})=\displaystyle{-\dfrac{1}{c^4}\left(\textbf{c},\mbox{rot}[f_{\textbf{x}}\times
\textbf{c}]\right)}=-\dfrac{1}{rc^2},$  we obtain that  the second
order differential system \eqref{E1} in this case takes the form
\[\ddot{{\textbf{x}}}=\dfrac{1}{c^2}\left(\partial_{\textbf{x}}(\dfrac{f}{r})-\dfrac{\partial _{\textbf{x}}f}{r}\right)=
-\dfrac{f}{c^2}\displaystyle{\dfrac{
\textbf{x}}{r^3}}=-\displaystyle{\dfrac{ \textbf{x}}{r^3}}\]

Below we study the particular case when the vector field $\textbf{a}$ satisfies the equation
\begin{equation}\label{K0}
\mbox{rot}(\textbf{a})=\nu (x)\textbf{a},\end{equation} where
$\nu$ is certain function.

\smallskip

{\it Optical-mechanical analogy}

\smallskip

From the standpoint of geometric optics, propagation of light in $
\mathbb{E}^3$ can be represented as a flow of particles.
Trajectories of particle are called { \it rays.} It is known
\cite{Koz1} that the vector field $ \textbf{K}$ of an arbitrary
system of rays an a homogeneous optical medium satisfies the
relation
$ \textbf{K}\times \mbox{\mbox{rot}}\textbf{K}=\textbf{0}.$
System of rays such that
$\mbox{\mbox{rot}}\textbf{K}\ne\textbf{0}$ are called {\it Kummer
systems}

\begin{proposition}
Let $ \textbf{a}$ be the Kummer vector field i.e.,  satisfies the
partial differential equation \eqref{K0} with $ \nu\ne 0.$
 Then Cartesian and Lagrangian approach generate the
following differential equations respectively
\begin{equation} \label{8}
\dot{\textbf{x}}=\textbf{V}-\gamma\textbf{a}=\dfrac{1}{||\textbf{a}||^2}[\textbf{a}\times\mbox{rot}\textbf{W}]=\textbf{v}(x)
\end{equation}
and
\begin{equation} \label{88}
\qquad\ddot{{\textbf{x}}}=\partial_{\textbf{x}}\left(\frac{1}{2}||[\textbf{a}\times
{\textbf{w}}]||^2\right)+\left(\textbf{R},\,[\textbf{
a}\times\textbf{w}])\right) {\textbf{a}},
\end{equation}
 where $\gamma=\dfrac{( \textbf{a},\textbf{V}
)}{|| \textbf{a}||^2},$\, $\textbf{V}$ and $\textbf{R}$ are the
vector fields:
\begin{equation} \label{9}
\mbox{rot}\textbf{V}=\nu\textbf{V}+[\textbf{a}\times\tilde{\textbf{V}}],\qquad \textbf{R}=\tilde{\textbf{V}}-\partial_{\textbf{x}}\gamma ,
\end{equation}
where $ \tilde{\textbf{V}}$ is an arbitrary smooth vector field.
\end{proposition}
\begin{proof}
Indeed, taking the well-known relations

\[
div[\textbf{A}\times\textbf{{B}}]=(\textbf{{A}},\mbox{rot}\textbf{{B}})-(\mbox{rot}\textbf{{A}},
\textbf{B}).
\]

 into account, we obtain
\[
 \mbox{div}[\textbf{a}\times\,[\textbf{a}\times\textbf{w}]]=\big(\textbf{a},\mbox{rot}[\textbf{a}\times\textbf{w}]\big)-
 \big(\mbox{rot}\textbf{a},[\textbf{a}\times\textbf{w}]\big)
 \]
Hence in view of this identity, and  considering that
$(\textbf{a},\mbox{rot}[\textbf{a}\times \textbf{w}])=0$ and
\eqref{K0}
 we obtain that $[\textbf{a}\times
[\textbf{a}\times{\textbf{w}}]]$ is a solenoidal vector field, so
\begin{equation}\label{V}
[\textbf{a}\times
[\textbf{a}\times\textbf{w}]]=(\textbf{a},\textbf{a})\textbf{w}-(\textbf{a},\textbf{w})
\textbf{a} =\mbox{rot}\textbf{W},
\end{equation}

where ${\textbf{W}}$ is a  smooth vector field  such that
$
(\textbf{a},\,\mbox{rot}\textbf{W})=0,
$
 consequently
 $
 \mbox{rot}\textbf{W}=[\textbf{a}\times{\textbf{V}}],
 $
for a smooth vector field $\textbf{V}$ which must satisfy the  partial differential equation
 \[\begin{array}{ll}
 \mbox{div}(\mbox{rot}\textbf{W})=
  \mbox{div}([\textbf{a}\times{\textbf{V}}])\\
  =(\textbf{a},\mbox{rot}\textbf{V})-(
 \textbf{V},\mbox{rot}\textbf{a})=\left(\textbf{a},\mbox{rot}\textbf{V}-\nu\textbf{V}\right)=0,
\end{array}
\]
hence we obtain the representation \eqref{9}

\smallskip

 In view of relation \eqref{V} we get the following representation
 for $\textbf{w}$
\[
\textbf{w}==\dfrac{(\textbf{a},\textbf{w})}{||\textbf{a}||^2}\textbf{a}+\frac{[\textbf{a}\times{\textbf{V}}]}{||\textbf{a}||^2}=
\dfrac{(\textbf{a},\textbf{w})}{||\textbf{a}||^2}\textbf{a}+\dfrac{\mbox{rot}\textbf{W}}{||\textbf{a}||^2}
\]
 From the relation $ \textbf{v}=[ \textbf{a}\times \textbf{w}]$ we
 get formula \eqref{8}.

\smallskip

On the other hand, after some calculations we obtain
\[\begin{array}{lll}
\mbox{rot}\textbf{v}=&\mbox{rot}\textbf{V}-[ \partial_{\textbf{x}}\gamma \times\textbf{a}]-\gamma\mbox{rot}\textbf{a}\vspace{0.20cm}\\
=&\nu \left(\textbf{V}-\gamma\textbf{a}\right)+[ \tilde{\textbf{V}}-\partial_{\textbf{x}}\gamma\times\textbf{a}]
\end{array}
\]or, what is the same
\begin{equation}\label{K22}
\mbox{rot} \textbf{v}=\nu \textbf{v}+[\textsc{R}\times\textbf{a}].
\end{equation}
Inserting \eqref{K22} in the relation $(\textbf{w},\mbox{rot}\textbf{v})$ we obtain
\[
\begin{array}{ll}
(\textbf{w},\mbox{rot}\textbf{v})&=(\textbf{w},\nu \textbf{v}+[\textbf{R}\times\textbf{a}])\\
&=(\textbf{w},\,[\textbf{R}\times\textbf{a}])=(-[\textbf{ w}\times\textbf{a}],\textbf{R})\\
&=\left([\textbf{ a}\times\textbf{w}], {\textbf{R}}\right).
\end{array}
\]
Hence from \eqref{E1} we obtain the second order differential equations \eqref{88}.
\end{proof}
\begin{corollary}
Let us suppose that Cartesian vector field is such that
\begin{equation}\label{K333}
{\mbox{rot}} \textbf{v}=\nu \textbf{v},\quad \nu\ne 0,
 \end{equation}
then Lagrangian approach generated the differential system
\[
\ddot{\textbf{x}
}=\partial_{\textbf{x}}\left(\dfrac{1}{2}||\textbf{v}||^2\right)
\]
which describe the motion of a material point of unit mass in the
potential field with the force function
$\dfrac{1}{2}||\textbf{v}||^2.$
\end{corollary}
\begin{proof}
{From} \eqref{K333} and \eqref{K22} follows that $
[\textsc{R}\times\textbf{a}]=\textbf{0},$ based on this relation
and considering that
  \[\left(\textbf{R},\,[\textbf{a}\times\textbf{w}])\right)=\left(\textbf{w},\,[\textbf{R}\times\textbf{a}])\right)\]

  we deduce that $\left(\textbf{R},\,[\textbf{a}\times\textbf{w}])\right)=0,$
 thus, in view of \eqref{88}, we obtain the proof of the proposition.

\smallskip

This results coincide with theorem  J.Bernoulli (1696.)  " Light
rays in an isotropic optical medium  with the refraction index
$n(x)$ coincide with the trajectories of a material point in a
potential field with the force function $U=\dfrac{1}{2}n^2(x).$"
\end{proof}

\smallskip

{ \it Example }

\smallskip

 Consider a particle with Lagrangian function
$L=\dfrac{1}{2}\big(\dot{x}^2+\dot{y}^2+\dot{z}^2\big)$ and
constraint  $\omega_2=\sin x\,\dot{y}-\cos x\,\dot{z}=0,$

\smallskip

The vector field  $ \textbf{a}=(0,\,\sin x,\,-\cos x),$
 satisfies the equation
$\mbox{rot} \textbf{a}=\textbf{a}.$

\smallskip

 Classical approach generates the differential equations
 \[ \ddot{x}=0,\quad \ddot{y}=\mu \sin x ,\quad \ddot{z}=-\mu \cos
 x .\]
 Integrating these equations we obtain
 \begin{equation}\label{LLLL10}
   \dot{x}=C_1,\quad \dot{y}=C_2\cos x,\quad \dot{z}=C_2\sin
 x,
 \end{equation}
 where $C_1$ and $C_2$ are arbitrary constants.  The vector field associated to this system is Cartesian type.

\smallskip

Cartesian approach generate the differential system
\[
 \dot{x}=\lambda ,\quad \dot{y}=\varrho\,\cos x ,\quad
\dot{z}=\varrho\,\sin x
\]
where $ \lambda=w_3\sin x+w_2\cos x,\,\varrho=-w_1,$ and
$w_1,\,w_2,\,w_3$ are components of the arbitrary vector field
$\textbf{ w}.$

\smallskip

From the condition $(\textbf{a},\mbox{rot}[\textbf{a}\times \textbf{w}])=0$ follows that the functions $\lambda$ and $\varrho$ are
solutions of the linear partial differential equation
\begin{equation} \label{u2}
\partial_z\lambda\sin x+\partial_y\lambda\cos x-\partial_x\varrho=0.
\end{equation}

 In particular the vector field
 \begin{equation}\label{K33}
\textbf{v}=\left(\lambda(x),\,\varrho (y,z)\cos
x,\,\varrho (y,z)\sin x\right),
\end{equation}
satisfies \eqref{u2},
where $\lambda=\lambda (x)$  and $\varrho=\varrho (y,z)$ are arbitrary real
functions, .

\smallskip

Solving equations  $ \lambda (x)=w_3\sin x+w_2\cos x,\,\varrho (y,z)=-w_1,$ with  respect to the components of the vector $\textbf{w}$   we obtain
\[\textbf{w}=\left(-\varrho (y,z),\,\lambda (x)\cos x-\psi (x,y,z)\sin x ,\,\lambda (x)\sin x+ \psi (x,y,z)\cos x\right)\]
 where $\psi=\psi (x,y,z)$ is an arbitrary function on $\textsc{Q}$.

 \smallskip

It is easy to check that
 \[\mbox{rot} \textbf{v}=\textbf{v}+\textbf{u},\]
 where $ \textbf{u}=(\partial_z\varrho \cos x-\partial_y\varrho\sin x-\lambda,0,0)^T,$ is a vector orthogonal to
 vector $\textbf{a}.$ Hence we obtain the relations

 \[
 [\textbf{R}\times \textbf{a}]=\textbf{u},\quad (\textbf{w},[\textbf{R}\times \textbf{a}])=\varrho\left(\partial_y\varrho\sin x+\lambda-\partial_z\varrho \cos x\right), \quad ||[\textbf{w}\times\textbf{a}]||^2=\varrho^2+\lambda^2.
 \]

\smallskip

Lagrangian approach generates the equations
\[ \ddot{\textbf{x}}=\partial_\textbf{x}\left(\dfrac{1}{2}( \varrho^2+\lambda^2)\right)-\varrho\left(\partial_z\varrho \cos x-\partial_y\varrho\sin x-\lambda\right)\textbf{a} \]

 \smallskip

 The vector field \eqref{LLLL10} can be obtained from \eqref{K33} if we choose  $\lambda=C_1,\,\varrho=C_2.$ For these parameter values
Lagrangian approach generates the differential equations
$ \ddot{\textbf{x}}=C_1C_2\textbf{a}.$

\smallskip
If we choose $ \varrho$ and $\lambda$: $(\partial_z\varrho \cos x-\partial_y\varrho\sin x-\lambda=0$
 Then $ \mbox{rot} \textbf{v}=\textbf{v}$ and as a consequence Lagrangian
 approach generate the equation $\ddot{\textbf{x}}=\partial_\textbf{x}\left(\dfrac{1}{2}( \varrho^2+\lambda^2)\right).$

\smallskip

Below we study the case when the constraints are integrable.

\begin{corollary}
Let $\textbf{V}$ and $\textbf{a}$ are the vector field such that
\begin{equation}\label{GG1}
 \mbox{rot}\textbf{a}=\textbf{0},\quad \textbf{V}=\partial_fG\partial_\textbf{x}\Phi
 \end{equation}

 where $G=G(f,\Phi )$ and $ \Phi$ are an arbitrary smooth functions,  then
  Cartesian and Lagrangian approach for a particle in $\mathbb{E}^3$
which is constrained to move on the surface $f=f(x)=c$ generate
the following differential equations
\begin{equation} \label{12}
\dot{\textbf{x}}=\partial_fG\left(\partial_\textbf{x}\Phi-\gamma\partial_\textbf{x}f\right)=\textbf{v}(x),
\end{equation}

and
\begin{equation} \label{K444}
\begin{array}{cc}
\ddot{\textbf{x}}=\partial_\textbf{x}\left(\dfrac{\partial_fG^2}{2}||\partial_\textbf{x}f\times\partial_\textbf{x}\Phi||^2\right)\vspace{0.20cm}\\
-\partial_f G\left(\partial^2_{ff}G\partial_\textbf{x}\Phi -
\partial_\textbf{x}\gamma\,,\partial_{\textbf{x}}\Phi-
\gamma\partial_\textbf{x}f\right)\partial_\textbf{x}f,
\end{array}
\end{equation}

where $g=||\partial_\textbf{x}f||^2,$ and $ \gamma=\dfrac{(\partial_\textbf{x}f,\,\partial_\textbf{x}\Phi)}{g},$

\smallskip

If\, $\gamma=0$ then the equations \eqref{12} and \eqref{K444}
take the form respectively
\begin{equation}\label{15}
\dot{\textbf{x}}=\partial_fG\partial_\textbf{x}\Phi ,\quad\ddot{\textbf{x}}=\partial_\textbf{x}
\left(\dfrac{G^2_f}{2}||\partial_\textbf{x}\Phi ||^2\right)-\left(G_{ff}G||\partial_\textbf{x}\Phi ||^2\right)\partial_\textbf{x}f
\end{equation}
\end{corollary}
\begin{proof}
Suppose that \eqref{GG1} hold, then \[\textbf{a}=\partial_\textbf{x}f,\quad
\mbox{rot}\textbf{V}=[\partial_\textbf{x}f\times\left(\partial^2_{ff}G\partial_\textbf{x}\Phi\right)=[\partial_\textbf{x}f\times\tilde{V} ],
\],
hence the vector fields $\textbf{W},\,\tilde{\textbf{V}}$ and $\textbf{ R}$ admit the representation
\[
\textbf{W}=G\partial_\textbf{x}\Phi,\quad\tilde{\textbf{V}}=\partial^2_{ff}G\partial_\textbf{x}\Phi+\upsilon\partial_\textbf{x}f ,\quad \textbf{R}=\partial^2_{ff}G\partial_\textbf{x}\Phi+
\upsilon\partial_\textbf{x}f -\partial_\textbf{x}\gamma,
\]where $\upsilon$ is an arbitrary function, hence we easily
obtain \eqref{K444}.

\smallskip

If\, $ \gamma=0$\, then $\left(\textbf{R},\partial_fG\left(\partial_\textbf{x}\Phi-\gamma\partial_\textbf{x}f\right)\right)=\partial_fG\partial^2_{ff}G||\partial_\textbf{x}\Phi||^2,$
thus formula \eqref{15} follow .
\end{proof}

\section{Integrability of the geodesic flow on the surface. }

\smallskip

It is well known that the differential equations
$
 \ddot{{\textbf{x}}}=\mu\,\partial_\textbf{x}f,
 $
 where $ \mu$ is Lagrangian multiplier, determine the geodesic flows on the surface $f=c,$
and admits the energy integral
$$||\dot{\textbf{x}}||^2=2h(f).$$
If there is an additional first integral, functionally independent
with the energy integral , then the {\it geodesic flow is integrable.}
\begin{proposition}

Let us suppose that \eqref{GG1} holds. Then Lagrangian geodesic
flow of the constrained particle on the surface $f=c,\quad
g=||\partial_\textbf{x}f||^2>0$ is integrable if there exist a
solution of the following non-lineal partial differential equation
respectively
\begin{equation}
\label{R}
{G^2_f}(f,\Phi\,)||[\partial_\textbf{x}f\times{\partial_\textbf{x}\Phi}]||^2=2h(f)g,
\quad\mbox{if}\quad  (\partial_\textbf{x}f,\partial_\textbf{x}\Phi)\ne{0},
\end{equation}
 and
\begin{equation}
\label{RR1}
{G^2_f}(f,\Phi\,)||\partial_\textbf{x}{\Phi}||^2=2h(f),\quad \mbox{if}
\quad (\partial_\textbf{x}f,\,\partial_\textbf{x}\Phi)=0,\quad \partial_\textbf{x}\Phi\ne \kappa \textbf{x}.
\end{equation}
\end{proposition}
\begin{proof}
Indeed, if \eqref{R} holds then the system \eqref{K444}takes the
form
\begin{equation} \label{K44}
\ddot{\textbf{x}}=\left(\partial_fh(f)-
\partial_f G\left(\partial^2_{ff}G\partial_\textbf{x}\Phi -
\partial_\textbf{x}\gamma\,,\partial_{\textbf{x}}\Phi-
\gamma\partial_\textbf{x}f\right)\right)\partial_\textbf{x}f,
\end{equation}
which determine the geodesic Lagrangian flow and admits the following complementary  first integral
\[F_1=\dfrac{g||[\partial_\textbf{x}\Phi\times{\dot{\textbf{x}}}]||^2}{(\partial_\textbf{x}f,\,\partial_\textbf{x}\Phi)^2}=C_1,\]
which it is easy to obtain from \eqref{12}. Clearly this first integral is independent of energy integral.

\smallskip

  If\, $(\partial_\textbf{x}f,\,\partial_\textbf{x}\Phi )={0}$ then Cartesian and Lagrangian
approach generated the differential equations respectively
\[
\dot{x}=G_f(f,\Phi\,)\partial_\textbf{x}\Phi,\quad
\ddot{x}=\big(h_f-G_f(f,\Phi\,)G_{ff}(f,\Phi\,)||{\partial_\textbf{x}\Phi}||^2\big)\partial_\textbf{x}f,
\]
The condition \eqref{R} under the given condition of orthogonality
takes the form  \eqref{RR1}.

\smallskip

 The complementary first integral is
\[
F_2=\dfrac{||\Phi_\textbf{x}||^2 ||[\textbf{x}\times
\dot{\textbf{x}}]||^2}{||[\textbf{x}\times{\partial_\textbf{x}\Phi}]||^2}=C_2,
\]which we can obtain from \eqref{12}, by considering that $ \gamma=0.$
This complete the proof of the proposition.
\end{proof}
Now we apply the above results.

\smallskip

 Now we consider the surface
\begin{equation}\label{16}
     f(x)=c, \quad  (\textbf{x},\,\partial_\textbf{x}f)=mf,\quad c\ne{0}\\
\end{equation}
 which we call
{\it homogeneous surface of degree $m.$}

\smallskip

We are interested in studying the integrability of Lagrangian geodesic flow
on the homogenous surface.

\smallskip

Euler's formula shows that $c=0$ is the unique critical
value of $f,$ hence for $c\ne{0}$ the function
$
g=||\partial_\textbf{x}f||^2>0,
$
on the surface $f(x)=c.$

\smallskip

 Taking into account  formula \eqref{16} it follows that
\[ (\textbf{x},\,\partial_\textbf{x}g)=2(m-1)g.
\]
Below we use the notation
\[
\{F,G,H\}=\left|
\begin{array}{rrr}
\partial_{x}F&\partial_{y}F&\partial_{z}F\\
\partial_{x}G&\partial_{y}G&\partial_{z}G\\
\partial_{x}H&\partial_{y}H&\partial_{z}H\end{array}\right|.
\]
Clearly, if $F,\,G,\,H$ are independent functions then
$\{F,G,H\}\ne{0}.$

\smallskip

 The integrability of
the geodesic flow on the homogeneous surface we  study in the
following two cases
\begin{equation}
\label{18} \{f,g,r^2\}=0, \quad \{f,g,r^2\}\ne 0 ,
\end{equation}
where $r^2=x^2+y^2+z^2.$

\smallskip

 We analyze the first case. We study only the particular subcase when
the homogeneous surface satisfies the condition
\begin{equation}
\label{19}
g=g(f,r).
\end{equation}
Hence, in view of \eqref{16} we obtain
$mf\partial_fg+r\partial_rg=2(m-1)g.$

\smallskip

 We assume that  function $\Phi$ is such that  $\Phi_\textbf{x}=\textbf{x},$
thus the differential equation generated by the Cartesian and Lagrangian approach
are respectively
\begin{equation}
\label{20}
\dot{\textbf{x}}=\dfrac{G_f}{g}\big(g\textbf{x}-m\,f\partial_\textbf{x}f\big),\quad
\ddot{\textbf{x}}=\frac{m\partial_rg f
h(f)}{r^2g^2}\partial_\textbf{x}f,\end{equation}

 \begin{proposition}

 Lagrangian geodesic  flow on the homogeneous surface under the
assumption \eqref{19} is integrable
\end{proposition}

\smallskip

\begin{proof}
Let us suppose that $\partial_\textbf{x}\Phi=\textbf{x},$ then $(\partial_\textbf{x}\Phi,\,\partial_\textbf{x}f)=mf\ne 0.$ On the other hand if we choose $G(f,r)$ as
\[
G^2_f(f,r )=\dfrac{2h(f) g(f,r)}{g(f,r)r^2-m^2f^2},
\] then we obtain that \eqref{R} holds. Thus there  exist an additional first
integral
\[g(f,r)||[\textbf{x}\times\dot{{\textbf{x}}}]||^2={m^2f^2}h(f)\]
\end{proof}
{\it Example }

\smallskip

 Lagrangian geodesic  flow on the homogeneous
surface of degree one
\[f(x)=r+(\textbf{b},\textbf{x})=c,\quad c\ne{0}\]
is integrable, where $ \textbf{b}$ is a constant vector field.
\begin{proof}
In this case we have $ g=\dfrac{2f}{r}+||\textbf{b}||^2-1=g(f,r).$

\smallskip

The complementary first integrals  is
\[ \big(\dfrac{2f}{r}+||\textbf{b}||^2-1\big)||[\textbf{x}\times\dot{\textbf{x}}]||^2=2f^2h(f).\]
\end{proof}
We have  studied the case in which $\{f,g,r^2\}=0.$  Now we
study the case in which the functions $f,\,g,\,r^2$ are
independent, i.e., $  \{f,g,r^2\}\ne{0}. $ Hence, we obtain that \begin{equation}\label{RRR1}
x=x(r,f,g),\quad y=y(r,f,g),\quad z=z(r,f,g).
\end{equation}

\smallskip

To establish the integrability or non-integrability of the
Lagrangian geodesic flow on the surface in this case it is
necessary to determine de existence or non-existence of the
solution of the equation \eqref{R} or \eqref{RR1}. We illustrate this case
 for the third-order surface
  \begin{equation}\label{210}
f(x)=x\,y\,z=c,\quad c\ne{0}.
\end{equation}
First we determine the dependence $ x=x(r,f,g),\,
y=y(r,f,g),\,z=z(r,f,g).$  By considering that in this case
$$g=(x\,y)^2+(x\,z)^2+(y\,z)^2$$
thus the functions $f,g$ and $r^2$ are independent. Indeed if we introduce the cubic
polynomial in $Z:$
$$P(z)=Z^3-r^2Z^2+g Z-f^2=(Z-x^2)(Z-y^2)(Z-z^2),$$
then by using  Cardano's formula we obtain the dependence \eqref{RRR1}.

\smallskip

This case was examined already by Riemann in his study of motion
of a homogeneous liquid ellipsoid. More exactly, Riemann examined
the integrability of the geodesic flow on \eqref{210}.

\smallskip

In \cite{Koz2} the author raised the problem.

\smallskip

"Is it true that the geodesic flow on a generic third-order
algebraic surface  is not integrable?. In particular I do not know
a rigorous proof of non-integrability for the surface \eqref{210}"

\smallskip

  To prove the integrability of Lagrangian geodesic flow it is
  necessary to solve the non-lineal partial differential
equation \[G_f(f,\Phi)\big( g||\Phi_\textbf{x}||^2-\big(
y\,z\Phi_x+z\,x\Phi_y+x\,y\Phi_z\big)^2\big)=2h(f)g\]

 \smallskip

 Now we are not able to provide an
answer to this question.

\section{ Decartes approach for non-holonomic system with four and five degree of freedom and two constraints}

In this section we apply Theorem \ref{main} to study the
non-holonomic system study in \cite{Gant} and the well known
non-holonomic system-the rattleback .

\smallskip

{\it Gantmacher's system}
\smallskip

Two material points $M_1,\,M_2$ with equal mass are linked by a metal rod with fixed long and small mass.
The system can move only in the vertical plane and so that the speed of the midpoint of the rod is directed along the rod.
It is necessary to determine the trajectories of material points $M_1,\,M_2.$

 \smallskip

Let $(x_1,\,y_1)$ and $(x_2,\,y_2)$ are the coordinates of the points $M_1,\,M_2.$

Introducing the following change of coordinates $$ 2u_1=x_2-x_1,\quad 2u_2=y_1-y_2,\quad 2u_3=y_2+y_1,\quad 2u_4=x_1+x_2$$ we obtain the  mechanical system
with configuration space $\textsc{Q}=\mathbb{R}^4,$ and Lagrangian function
$
L=\displaystyle\frac{1}{2}\sum_{j=1}^4\dot{u}^2_j-gu_3.
$

 \smallskip

 The equations of the constraints can be rewritten as
\[u_1\dot{u}_1+u_2\dot{u}_2=0,\quad u_1\dot{u}_3-u_2\dot{u}_4=0.\]

 To construct Cartesian approach in this case we firstly  determine  the
1-forms $\Omega_j$ for $j=1,2,3,4$ as follow
\[
\begin{array}{ll}
\Omega_1=u_1d{u}_1+u_2d{u}_2,\quad \Omega_2=u_1d{u}_3-u_2d{u}_4,\\
 \Omega_3=u_1du_2-u_2du_1,\quad \Omega_4=u_2du_3+u_1du_4
 \end{array}
 \]
 hence we obtain that $\Upsilon=u^2_1+u^2_2.$

 \smallskip

After some calculations we obtain that the vector field \eqref{00} takes the form
$$ \textbf{v}=\nu_3(u_1\partial_2-u_2\partial_1)+\nu_4(u_2\partial_3+u_1\partial_4),\quad \Upsilon=(u^2_1+u^2_2)^2,\quad
$$
 where $\nu_j=\lambda_j (u^2_1+u^2_2)$ for $j=3,4.$

 \smallskip

 The 1-form associated to vector $\textbf{v}$ is the following
 \[ \sigma=\nu_3(-u_2du_1+u_1du_2)+\nu_4(u_2du_3+u_1du_4).\]
 Thus the 1-form $ \iota_{\textbf{v}}d\sigma$ admits the representation
  \[
 \begin{array}{ll}
 \iota_{\textbf{v}}d\sigma=\Lambda_1\Omega_1+\Lambda_2\Omega_2+\Lambda_3\Omega_3+\Lambda_4\Omega_4\vspace{0.20cm}\\
 =(u^2_1+u^2_2)\Big( -u_2\partial_{u_2}(\dfrac{\nu^2_3+\nu^2_4}{2}) -u_1\partial_{u_1}(\dfrac{\nu^2_3+\nu^2_4}{2})-2\nu^2_3-\nu^2_4\Big)\Omega_1
 \vspace{0.20cm}\\
 + (u^2_1+u^2_2)\Big( u_2\partial_{u_4}(\dfrac{\nu^2_3+\nu^2_4}{2}) -u_1\partial_{u_3}(\dfrac{\nu^2_3+\nu^2_4}{2})+\nu_3\nu_4\Big)\Omega_2
 \vspace{0.20cm}\\
   +(u^2_1+u^2_2)\nu_4\Big( u_2\partial_{u_3}\nu_3+u_1\partial_{u_4}\nu_3+u_2\partial_{u_1}\nu_4- u_1\partial_{u_2}\nu_4 \Big) \Omega_3
   \vspace{0.20cm}\\
 +(u^2_1+u^2_2)\nu_3\Big( u_2\partial_{u_3}\nu_3+u_1\partial_{u_4}\nu_3+u_2\partial_{u_1}\nu_4-
 u_1\partial_{u_2}\nu_4 \Big)\Omega_4.
 \end{array}
 \]
If $ \Lambda_3=\Lambda_4=0$ then\, $ u_2\partial_{u_3}\nu_3+u_1\partial_{u_4}\nu_3+u_2\partial_{u_1}\nu_4-
 u_1\partial_{u_2}\nu_4 =0.$

 \smallskip

  Cartesian  approach  generate the following differential equations respectively
 \begin{equation}\label{G100}
 \dot{u}_1=-\nu_3u_2,\quad\dot{u}_2=\nu_3u_1,\quad\dot{u}_3=\nu_4u_2,\quad\dot{u}_4=\nu_4u_1
 \end{equation}
and
\begin{equation}\label{G2}
 \quad  u_2\partial_{u_3}\nu_3+u_1\partial_{u_4}\nu_3+u_2\partial_{u_1}\nu_4-
 u_1\partial_{u_2}\nu_4=0
\end{equation}
 It is easy to show that the functions $\nu_3,\,\nu_4:$
\begin{equation}\label{G30}
\nu_3=g_3(u^2_1+u^2_2),\qquad \nu_4=\sqrt{\frac{2(-gu_3+h)}{(u^2_1+u^2_2)}- g^2_3(u^2_1+u^2_2)},
   \end{equation}
where $g,\,h$ are constants, are solutions of
\eqref{G2},\,\eqref{G100}as a consequence
$$2||\textbf{v}||^2=(u^2_1+u^2_2)(\nu^2_3+\nu^2_4)=2(-gu_3+h).$$
Under these restrictions  Lagrangian approach generate the differential system
\begin{equation}\label{G3}
\ddot{u}_1=\Lambda_1u_1,\quad \ddot{u}_2=\Lambda_1u_2,\quad \ddot{u}_3=-g+\Lambda_2u_1,\quad \ddot{u}_4=\Lambda_2u_2.
\end{equation}
The solutions of \eqref{G2} are
\[\begin{array}{ll}
u_1=r\cos{\alpha},\quad u_2=r\sin{\alpha},\quad\alpha=\alpha_0+g_3(r)t,\vspace{0.20cm}\\
u_3=u^0_3+\displaystyle\dfrac{g}{2g_3(r)}t-\displaystyle\dfrac{g}{4g_3^2(r)}\sin2\alpha-
-\displaystyle\dfrac{\sqrt{2g}C}{g_3(r)}\cos\alpha ,\vspace{0.20cm}\\
u_4=-h+\displaystyle\dfrac{r^2g_3^2(r)}{2g}+
\Big(\displaystyle\dfrac{\sqrt{g}}{\sqrt{2} g_3(r)}\sin{\alpha}+C)^2,\vspace{0.20cm}\\
\end{array}
\]
where $C,\,r,\, \alpha_0,\,u^0_3,\,h,\,$ are arbitrary constants,
$g_3$ is an arbitrary on  $r$ function.

 \smallskip

To compare these solutions with the solutions  obtained from the classical
approach, we determine the equations of motion obtained
from the d'Alembert-lagrange principle
\begin{equation}\label{G0}
\ddot{u_1}=\mu_1u_1,\quad\ddot{u_2}=\mu_1u_2,\quad\ddot{u_3}=-g+\mu_2u_1,\quad\ddot{u_4}=-\mu_2u_2
\end{equation}
 where $\mu_1,\,\mu_2$ are the Lagrangian multipliers.

 \smallskip

After the integration of the system \eqref{G0} we obtain
\cite{Gant}
\begin{equation}\label{G31}
\dot{u}_1=-\dot{\varphi}u_2,\quad\dot{u}_2=\dot{\varphi}u_1,\quad\dot{u}_3=\frac{f}{r}u_2,\quad\dot{u}_4=\frac{f}{r}u_1
\end{equation}

where $(\varphi,\,r)$ are the polar coordinates: $u_1=r\cos\varphi,\quad u_2=r\sin\varphi$
and $f$ is a solution of the equation
\begin{equation}\label{G4}\dot{f}=-\frac{2g}{r}u_2\end{equation}
The solution of \eqref{G4} is
$f=\dfrac{2g\,\cos\varphi}{\dot{\varphi}}+2\gamma
 $ where $\gamma$ is an arbitrary constants.

 \smallskip

 Clearly if we choose $\nu_3=\dot{\varphi},\quad \nu_4={\dfrac{f}{r}}$
 then we the vector field associated to system \eqref{G31} can be obtained from Cartesian vector field.

\smallskip

{\it The rattleback.}

\smallskip

 The rattleback's  amazing mechanical behaviour is a
convex asymmetric rigid body rolling without sliding on a
horizontal plane. It is known for its ability to spin in one
direction and to resist spinning in the opposite direction for
some parameters values, and for others values to exhibit multiple
reversals. Basic references on the rattleback are \cite{Whal, Gar, Kar,
Bor}.

Introduce the Euler angles $\psi ,\,\phi ,\,\theta $ using the
principal axis body frame relative to an inertial reference frame.
These angles together with two horizontal coordinates $x,\,y$ of
the center of mass are coordinates in the configuration space
$\textsc{Q}=SO(3)\times\mathbb{R}^2$ of the rattleback.

The Lagrangian of the rattleback is computed to be
\[ \begin{array} {ll}
L=&\frac{1}{2}(I_1\cos^2\psi+I_2\sin^2\psi+m(\Gamma_1\cos\theta-\zeta\sin\theta)^2)\dot{\theta}^2\vspace{0.20cm}\\
&\frac{1}{2}(I_1\sin^2\psi+I_2\cos^2\psi) \sin^2\theta
)+I_3\cos^2\theta )\dot{\phi}^2\vspace{0.20cm}\\
&+\frac{1}{2}(I_3+m\Gamma^2_2\sin^2\theta
)\dot{\psi}^2+\frac{m}{2}(\dot{x}^2+\dot{y}^2)\vspace{0.20cm}\\
&+m(\Gamma_1\cos\theta-\zeta\sin\theta)\Gamma_2\sin\theta\dot{\theta}\dot{\psi}+
(I_1-I_2)\sin\theta\sin\psi\cos\psi\dot{\theta}\dot{\phi}\vspace{0.20cm}\\
&C\cos\theta\dot{\phi}\dot{\psi}+mg(\Gamma_1\sin\theta
+\zeta\cos\theta )
\end{array}
\] where $I_1,I_2,I_3$ are the
principal moments of inertia of the body, $m$ is the total mass of
the body,
$$\Gamma_1=\xi \sin\psi+\eta \cos\psi,\quad\Gamma_2=\xi \cos\psi-
\eta \sin\psi$$ and $\left(\xi=\xi (\theta ,\psi ),\,\eta=\eta (\theta ,\psi )
,\,\zeta=\zeta (\theta ,\psi ) \right)$ are the coordinates of the point of
contact relative to the body frame.The shape of the body is encoded by the functions $\xi,\,\eta$ and
$\zeta .$

\smallskip

The constraints are
\[
\dot{x}-\alpha_1\dot{\theta}-\alpha_2\dot{\psi}-\alpha_3\dot{\phi}=0,\quad\dot{y}+\beta_1\dot{\theta}+\beta_2\dot{\psi}+\beta_3\dot{\phi}=0,
\]
where
\[\begin{array} {cc}
\alpha_1=(-\Gamma_1\sin\theta-\zeta\cos\theta)\sin\phi,\quad\alpha_2=\Gamma_2\cos\theta\sin\phi +\Gamma_1\cos\phi,\\\alpha_3=\Gamma_2\sin\phi +(\Gamma_1\cos\theta-\zeta\sin\theta)\cos\phi,\quad
\beta_1=\dfrac{\partial \alpha_1}{\partial\phi},\,\beta_2=\dfrac{\partial \alpha_2}{\partial\phi},\,\beta_3=\dfrac{\partial \alpha_3}{\partial\phi}.
\end{array}
\]
 Clearly that the rattleback
equations of motion in this particular case
 formally contain the equations of the heavy rigid body in the
 singular case
 $m\rightarrow{0},\quad mg\rightarrow{l},\quad l\ne{0}$

 \smallskip

To determine Cartesian approach for the rattleback we first
determine the 1-forms $\Omega_j,$ for $ j=1,\ldots,5.$ In this
case we determine as follows
\[\begin{array} {cc}
&\Omega_1=dx-\alpha_1d{\theta}-\alpha_2d{\psi}-\alpha_3d{\phi},\quad \Omega_2=dy+\beta_1d{\theta}+\beta_2d{\psi}+\beta_3d{\phi},\vspace{0.20cm}\\
&\Omega_3=d\theta,\quad\Omega_4=d\psi,\quad \Omega_5=d\phi
\end{array}
\]
Hence  $\Upsilon=1$ and  the vector field $\textbf{v}:$
\begin{equation}\label{R2}
\textbf{v}=\lambda_3X_3+\lambda_4X_4+\lambda_5X_5,
\end{equation}
where
\[X_3=\alpha_1\partial_x-\beta_1\partial_y+\partial_\theta,\quad X_4=\alpha_2\partial_x-\beta_2\partial_y+\partial_\psi,\quad X_5=\alpha_3\partial_x-\beta_3\partial_y+\partial_\phi.
\]
 We now proceed to the consideration of the particular case for which $\xi,\,\eta$ and
$\zeta$ admits the development

\[ \xi={\xi}_0+\epsilon \xi_1(\theta,\,\psi ),\quad \eta=\eta_0+\epsilon \eta_1(\theta,\,\psi ),\quad \zeta=\zeta_0+\epsilon\zeta_1(\theta,\,\psi )\]

where $\xi_0,\,\eta_0,\,\zeta_0,$ are constants and $\epsilon$ is
a small parameter. Under this consideration we obtain that the
Lagrangian function can be represented as follow
\[L= L_0+\epsilon L_1+\epsilon^2 L_2.\]Below we study the case when $ \epsilon=0.$

\smallskip

Let $(x^1,\,x^2,\,x^3,\,x^4,\,x^5)$ be a new set of variables
derived from $x,,\,y,\,\theta,\,\psi,\,\phi$ by the transformation
\[\begin{array}{cc}
&\psi=x^1,\quad\phi=x^2,\quad\theta=x^3,\\
&y+\zeta_0\sin\theta\cos\phi+\Gamma^0_1\cos\theta\sin\phi-\Gamma^0_2\sin\phi=x^4,\\
&x+\zeta_0\sin\theta\sin\phi-\Gamma^0_1\cos\theta\cos\phi+\Gamma^0_2\cos\phi=x^5,
\end{array}
\]where $ \Gamma^0_1=\xi_0 \sin\psi+\eta_0 \cos\psi,\quad\Gamma^0_2=\xi_0 \cos\psi-
\eta_0 \sin\psi.$

\smallskip

The vector field $\textbf{v}$ and the constraints on account of this
change, take respectively the form respectively
\[\textbf{v}=\left(a,\,b,\,c,\,0,\,0\right)\quad\dot{x}^4=0,\quad\dot{x}^5=0
\]
where $ a=a(x^1,..,x^5)),\,b=b(x^1,..,x^5)),\,c=c(x^1,..,x^5))$ are the $ \mathcal{C}^1$ functions.

\smallskip

 In the coordinates $x=(x^1,\,x^2,\,x^3,\,x^4,\,x^5)$  the Lagrangian function $L_0$ becomes to the function
$$\tilde{L}=\frac{1}{2}\sum_{j,k=1}^5G_{jk}\dot{x}^j\dot{x}^k+mg(\Gamma^0_1\sin{x^3}
+\zeta_0\cos{x^3} ),$$ where $G=(G_{jk}(x))=(G_{jk})$ is the
Riemann metric.

\smallskip

We shall now determine Cartesian approach under the given
conditions.
\begin{proposition}

The vector field $\tilde{\textbf{v}}(x)=( a,\,b,\,c)$ is a
Kummer vector field.
\end{proposition}
\begin{proof}

 Indeed the 1-form
associated to the vector field $\textbf{v}$ is
\[
\sigma=p_1dx^1+p_2dx^2+p_3dx^3,\quad p_k=G_{k1}a+G_{k2}b+G_{k3}c,\quad
k=1,2,..,5
\]
then
\[
\imath_{\textbf{v}}d\sigma=\sum_{j=1}^5\Lambda_jdx^j
\]
where
\[
\begin{array}{ll}
& \Lambda_1=(\dfrac{\partial p_1}{\partial
x^2}-\dfrac{\partial p_2}{\partial x^1})b+(\dfrac{\partial
p_1}{\partial x^3}-\dfrac{\partial
p_3}{\partial x^1})c,\quad\Lambda_2=(\dfrac{\partial p_2}{\partial x^3}-\dfrac{\partial
p_3}{\partial x^2})c+(\dfrac{\partial
p_2}{\partial x^1}-\dfrac{\partial p_1}{\partial x^2})a\vspace{0.20cm}\\
&\Lambda_3=(\frac{\partial p_3}{\partial x^2}-\frac{\partial
p_2}{\partial x^3})b+(\dfrac{\partial
p_3}{\partial x^1}-\dfrac{\partial p_1}{\partial x^3})a,\quad\Lambda_4=-\dfrac{\partial p_1}{\partial x^4}a-\dfrac{\partial
p_2}{\partial x^4}b-\dfrac{\partial p_3}{\partial x^4}c\vspace{0.20cm}\\
&\Lambda_5=-\dfrac{\partial p_1}{\partial x^5}a-\dfrac{\partial
p_2}{\partial x^5}b-\dfrac{\partial p_3}{\partial
x^5}c
\end{array}
\]
We have therefore that the differential equations generated by  Cartesian approach are respectively
\begin{equation}\label{R5}
\begin{array}{ll}
\dot{x}^1=a,\quad\dot{x}^2=b,\quad\dot{x}^3=c,\vspace{0.20cm}\\
\Lambda_1=\displaystyle(\frac{\partial p_1}{\partial
x^2}-\frac{\partial p_2}{\partial x^1})b+(\frac{\partial
p_1}{\partial x^3}-\frac{\partial
p_3}{\partial x^1})c=0,\vspace{0.20cm}\\
\Lambda_2=\displaystyle(\frac{\partial p_2}{\partial x^3}-\frac{\partial
p_3}{\partial x^2})c+(\frac{\partial
p_2}{\partial x^1}-\frac{\partial p_1}{\partial x^2})a=0,\vspace{0.20cm}\\
\Lambda_3=\displaystyle(\frac{\partial p_3}{\partial x^2}-\frac{\partial
p_2}{\partial x^3})b+(\frac{\partial
p_3}{\partial x^1}-\frac{\partial p_1}{\partial x^3})a=0,\\
\end{array}
\end{equation}
where $a=a(x^1,x^2,x^3,C_4,C_5),\,b=b(x^1,x^2,x^3,C_4,C_5),\,c=c(x^1,x^2,x^3,C_4,C_5)$
Let  $\mbox{rot}{\tilde{\textbf{v}}(x)}$ be the vector field
\[rot{\tilde{\textbf{v}}}=\frac{1}{\sqrt{det{G}}}(\frac{\partial
p_3}{\partial x^2}-\frac{\partial p_2}{\partial
x^3},\,\frac{\partial p_1}{\partial x^3}-\frac{\partial
p_3}{\partial x^1},\,\frac{\partial p_2}{\partial
x^1}-\frac{\partial p_1}{\partial x^2})^T,
\]
then the last three equations in \eqref{R5} can be rewritten as
$$[{\tilde{\textbf{v}}(x)}\times{rot{\tilde{\textbf{v}}(x)}}]=\textbf{0}.$$

Thus the vector field $\tilde{\textbf{v}}(x)=(a,b,c)$  is a Kummer vector field.
\end{proof}

 For the general case,
i.e., when the $\xi ,\,\eta$ and $\zeta$ are functions on the
variables $\theta$ and $\psi$ Cartesian approach produce the
following equations respectively

$$\dot{\textbf{x}}=\textbf{v}(x)$$

\[\begin{array}{ccccc}
&\displaystyle\sum_{j=1}^5\Big(\frac{\partial p_1}{\partial x^j}-\frac{\partial
p_j}{\partial x^1}+\alpha_2(\frac{\partial p_4}{\partial
x^j}-\frac{\partial p_j}{\partial x^4})-\beta_2(\frac{\partial
p_5}{\partial x^j}-\frac{\partial p_j}{\partial
x^5})\Big)v^j=0,\vspace{0.20cm}\\
&\displaystyle\sum_{j=1}^5\Big(\frac{\partial p_2}{\partial x^j}-\frac{\partial
p_j}{\partial x^2}+\alpha_3(\frac{\partial p_4}{\partial
x^j}-\frac{\partial p_j}{\partial x^4})-\beta_3(\frac{\partial
p_5}{\partial x^j}-\frac{\partial p_j}{\partial
x^5})\Big)v^j=0,\vspace{0.20cm}\\
&\displaystyle\sum_{j=1}^5\Big(\frac{\partial p_3}{\partial x^j}-\frac{\partial
p_j}{\partial x^1}+\alpha_1(\frac{\partial p_4}{\partial
x^j}-\frac{\partial p_j}{\partial x^4})-\beta_1(\frac{\partial
p_5}{\partial x^j}-\frac{\partial p_j}{\partial
x^5})\Big)v^j=0
\end{array}
\]
where
$\psi=x^1,\quad \phi=x^2,\quad\theta=x^3,\quad y=x^4,\quad x=x^5
$ and
$\textbf{v}$ is given by the formula \eqref{R2}.

\smallskip

\section{Inverse problem of dynamics}

\smallskip

{\it Introduction}

\smallskip

  This section is devoted to apply Corollary \ref{main2} to study the problem of
finding the field of force that generates a given
($N-1$)-parametric family of orbits for a mechanical system with
$N$ degrees of freedom. This problem is usually referred to as the
inverse problem of dynamics. We study this problem in relation to
the problems of Celestial Mechanics.

One of the fundamental classical problems in celestial mechanics
is to  determine  the potential-energy function $U$ such that
every curve from a given family of curves will be a possible
trajectory of a particle moving under the action of potential
forces $\textbf{F}$, admitting $U$; i. e. $\textbf{F}=\displaystyle\dfrac{\partial
U}{\partial{\textbf{x}}}$.

In the modern scientific literature the importance of this problem was already acknowledged by
Szebehely \cite{Boz},\,\cite{Szbehely}

The first inverse  problem in  Celestial Mechanics was stated and
solved  by Newton (1687) and concerns the  determination of the
potential field of force  that ensures the planetary motion in
accordance to the observed properties, namely to Kepler's laws.

Bertrand (1877) \cite{Bertrand}  proved that the expression for Newton's
force of attraction can be obtained directly from  the Kepler
first law to within a constant multiplier.

Bertrand stated also a more general problem of determining a
positional force, under which a particle describes a conic section
under any initial conditions.
Bertarnd's ideas were developed by \cite{Dainelli} \cite{Sus}, \cite{ Joukovski},
\cite{Ermakov}, and \cite{Galiullin}.

Dainelli in \cite{Dainelli} essentially  states a more general problem of
how to determine the most general field of force (the force being
supposed to depend only on the position of the particle on which it
acts) under which a given family of planar curves is a family of
orbits of a particle.

The solution proposed by Dainelli is the following .

\begin{theorem}\label{Dainelli}
The most general field of force $\textbf{F}=({F}_x,\,F_y)$ capable of
generating the family of planar orbits $f(x,y)=const$ can be
determine as follows \cite{Dainelli},\, \cite{Whittaker}

\begin{equation}\label{21}
F_x=-\lambda^2\{f,\partial_yf\}-\lambda\{f,\,\lambda\}\partial_yf,\quad F_y=\lambda^2\{f,\partial_xf\}+\lambda\{f,\,\lambda\}\partial_xf,
\end{equation}
where $\{f,\quad\}=\partial_xf\partial_y- \partial_yf\partial_x $ and $\lambda$ is an arbitrary function which depends on the
velocity with which the given orbits are described.

\smallskip

 By considering
that the components $F_x$ and $F_y$ are to be functions of the
position of the particle, we can take $\lambda$ to be an arbitrary
function on $x$ and $y.$
\end{theorem}
The above expressions for the field of force under which the
curves of the given family are orbits were first given by Dainelli
\cite{Dainelli}.

After some calculations we can prove that \eqref{21} can be rewritten as follows
\begin{equation}\label{D1}
\textbf{F}=\displaystyle\dfrac{\partial
\dfrac{1}{2}\|\textbf{v}\|^2}{\partial{\textbf{x}}}-\lambda \left( \partial_x(\lambda \partial_xf)+\partial_y(\lambda \partial_yf)\right)\displaystyle\dfrac{\partial
f}{\partial{\textbf{x}}}
\end{equation}
where $\textbf{v}=\left(
-\lambda\partial_yf,\,\lambda\partial_xf\right).$

\smallskip

  Suslov in \cite{Sus} stated and solved a problem which was a
further development of Bertrand's problem. He shows that, given a
($N-1$)-parametric family of orbits $f_j=f_j(x)=c_j$ for
$j=1,2,\ldots N-1$ in the configuration space of a holonomic
system with $N$ degrees of freedom and a kinetic energy
$T=\displaystyle\dfrac{1}{2}\sum_{j,k=1}^NG_{jk}(x)\dot{x}^j\dot{x}^k=\displaystyle\dfrac{1}{2}||\dot{\textbf{x}}||^2,$
it is necessary to determine the potential field of force under
which any trajectory of the family can be traced by the
representative point of the system. Suslov deduced the following
system of linear partial differential equations with respect to
the require potential function:

\[\begin{array} {cc}
\dfrac{\partial\theta}{\partial\triangle_k}\dfrac{\partial
U}{\partial{x}^N}-\dfrac{\partial\theta}{\partial\triangle_N}\dfrac{\partial
U}{\partial{x}^k}&=
\dfrac{U+h}{\theta}\Big(\dfrac{\partial\theta}{\partial\triangle_N}\dfrac{\partial
\theta}{\partial{x}^k}-\dfrac{\partial\theta}{\partial\triangle_k}\dfrac{\partial
\theta}{\partial{x}^N}\vspace{0.20cm}\\&+\displaystyle\sum_{m=1}^N\triangle^m(\dfrac{\partial\theta}{\partial\triangle_k}\dfrac{\partial^2
\theta}{\partial\triangle_N\partial{x}^m}-\dfrac{\partial\theta}{\partial\triangle_N}\dfrac{\partial^2
\theta}{\partial\triangle_k\partial{x}^m})\Big)
\end{array}
\]
for  $k=1,2,..,N-1.$
  where
$\theta,\,\triangle^1,\,\triangle^2,...,\triangle^N$ are
functions:

\[\begin{array} {cc}
&\displaystyle\sum_{k=1}^N\frac{\partial
f(x)_\alpha}{\partial{x}^k}\triangle^k=0,\quad\theta=\frac{1}{2}\displaystyle\sum_{k,j=1}^NG_{kj}(x)\triangle^k\triangle^j\vspace{0.20cm}\\
&\triangle_k=\displaystyle\sum_{j=1}^NG_{jk}(x)\triangle^j,\quad
k=1,2,..,N,\, \alpha=1,2,..,N-1,
\end{array}
\]

 and proved that theses equations represented the necessary and
 sufficient conditions under which the equations of motion of the
 study mechanical system admits the given $N-1$ partial integrals.

Assuming that given trajectories admit a family of the orthogonal
surfaces, Joukovski in \cite{Joukovski} constructed the potential-energy
functions in explicit forms for systems with two and three degrees
of freedom.

The following theorem was enunciated by Joukovsky in 1890
\begin{theorem}\label{J1}
 If $q=const$ is the equation of the family of curves on a
surface, and $p=const$ denotes the family of curves orthogonal to
these, then the curves  $q=const$ can be freely described by a
particle under the influence of forces derived from the
potential-energy function
$$V=\Delta_1(p)\Big(g(p)+\int h(q)\frac{\partial}{\partial
q}(\frac{1}{\Delta_1(p)})dq\Big)$$ where $h$ and $g$ are arbitrary
functions, and $\Delta_1$ denotes the first differential
parameter.
\end{theorem}
 A new approach to the problem of constructing the
potential field of force was proposed by Ermakov in \cite{Ermakov}, who
integrated the equations for the potential-energy function for
several particular cases.

In the most general form the inverse problem in dynamics was
studied in \cite{Sad, Ram3}. By applying the results presented in that
work we propose the following new results:
\begin{itemize}
\item[(i)] Statement and solution of inverse  Dainelli's problem for a mechanical
system with $N$ degree of freedom.
\item[(ii)]  New approach to solve the Suslov problem.
\item[(iii)]Statement and solution of inverse Joukovski's problem for mechanical system with $N\geq{3}$ degree
of freedom.
\item[(iv)] Generalization of Theorem  \ref{J1} for mechanical system with $N\geq{3}$ degree
of freedom
\item[(v)]  Statement and solution of inverse St\"{a}ckel's problem.
\item[(vi)] General solution of Bertrand's  inverse problem.
\end{itemize}

\smallskip
\smallskip

The results listed above are obtained by applying Corollary \ref{main2}.

\smallskip
\smallskip

{\it Statement of the generalized inverse Dainelli problem.}

\smallskip
\smallskip

Given a $N-1$ -parametric family of orbits $f_j=f_j(x)=c_j$ for $j=1,2,\ldots,N-1$ in the configuration
space $\textsc{Q}$ of a holonomic system with $N$ degrees of freedom and
kinetic energy $T=\dfrac{1}{2}||\dot{\textbf{x}}||^2.$ {\it Generalized Dainelli's problem} is the
problem of determining the most general field of force that
depends only on the position of the system   under which any
trajectory of the family can be traced by a representative point
of the system.

\smallskip
\smallskip

{\it Solution of the generalized inverse Dainelli problem}

\smallskip
\smallskip

The following proposition provides a solution to the problem above

\begin{proposition}\label{D2}
  Given a mechanical system with a configuration space $\textsc{Q}$  and a kinetic energy $T=\displaystyle\dfrac{1}{2}||\dot{\textbf{x}}||^2$, then the
most general field of force $\textbf{F}$ that depends  only on the position of
the system and is capable of generating the given orbits $ f_j (
x) = c_j, $ for $j=1, \ldots, N-1 $ where $f_1, \ldots, f_{N-1}$ are independent functions can be determine from the formula
\begin{equation}\label{D30}
\textbf{F}=\displaystyle\dfrac{\partial
\left(\dfrac{1}{2}||\textbf{v}||^2\right)}{\partial{\textbf{x}}}+
\lambda\displaystyle\sum_{j=1}^{N-1}a_{Nj}\displaystyle\dfrac{\partial
f_j}{\partial{\textbf{x}}},
\end{equation}
where
\begin{equation}
\label{T1}
\textbf{v}=-\lambda_N\dfrac{\{f_1,\ldots,\,f_{N-1},\,*\}}{\{f_1,\ldots,\,f_{N-1},\,f_N\}}=\lambda\{f_1,\ldots,\,f_{N-1},\,*\},
\end{equation}
 $a_{Nj}=a_{Nj}(x)$ for $j=1,2,\ldots N$ are functions:
$$ a_{Nj} = (-1)^{N+j-1} d\sigma \wedge df_1\wedge df_2\wedge \ldots \wedge
df_{j-1} \wedge df_{j+1} \wedge \ldots \wedge df_{N-1}(\partial_1, \ldots,
\partial_N),$$ and $\lambda_N$ and $f_N$ are  arbitrary functions such that $ \{f_1,\ldots,\,f_{N-1},\,f_N\}\ne 0$.
\end{proposition}

\begin{proof}
Let us suppose that  are given the $N-1$ parametric family of trajectories, hence
$ ( \partial_\textbf{x} f_j,\,\dot{\textbf{x}})=0$ for $j=1,\ldots N-1.$  In view of independence of functions $f_1,\,\ldots,\,f_{N-1}$
we can solve these equations respect to velocity, thus we obtain the system
$
 \dot{\textbf{x}}=\textbf{v}(x),
 $
where $\textbf{v}$ is determine by the formula \eqref{T1}.

\smallskip

After covariant derivation we obtain the equations of motion of the mechanical system (see proof of Theorem \ref{main})

 \[
\displaystyle\frac{d}{dt}\displaystyle\frac{\partial
T}{\partial\dot{\textbf{x}}}-\displaystyle\dfrac{\partial
T}{\partial{\textbf{x}}} =\displaystyle\dfrac{\partial
\left(\dfrac{1}{2}||\textbf{v}||^2\right)}{\partial{\textbf{x}}}+\lambda\sum_{j=1}^{N-1}a_{Nj}\displaystyle\dfrac{\partial
f_j}{\partial{\textbf{x}}}=\textbf{F},
\]
therefore the proposition has been proved.
\end{proof}
The following proposition shows that Theorem \ref{Dainelli} is a particular case of Proposition \ref{D2}.

\begin{corollary}\label{LL30}
For $N=2$ and $\textsc{Q}=\mathbb{R}^2$  the force field
\eqref{D30} coincides with the solution proposed by Dainelli .
\end{corollary}
\begin{proof}
Indeed for $N=2$ the field of force $\textbf{F}$ takes the form
\[\textbf{F}=\displaystyle\frac{\partial
\left(\dfrac{1}{2}||\textbf{v}||^2\right)}{\partial{\textbf{x}}}+\lambda
\, a_{21}\displaystyle\frac{\partial f}{\partial{\textbf{x}}}.\]

On the other hand by considering that  $\textbf{v}=\left(
-\lambda\partial_yf,\,\lambda\partial_xf\right)$ thus
\[|\textbf{v}||^2= (\lambda\partial_xf)^2+(\lambda\partial_yf)^2,\quad a_{21}=d\sigma(\partial_x,
\,\partial_y)=\partial_x(\lambda \partial_xf)+\partial_y(\lambda
\partial_yf)\] hence we obtain the formula \eqref{D1}.
\end{proof}
\begin{corollary}\label{SS33}
For $N=3$  the force field \eqref{D30} takes the form
\begin{equation}\label{DDD30}
\textbf{F}=\displaystyle\dfrac{\partial
\left(\dfrac{1}{2}||\textbf{v}||^2\right)}{\partial{\textbf{x}}}+
\lambda \left(df_1(\mbox{rot}{\textbf{v}})df_2-df_2(\mbox{rot}{\textbf{v}})df_1\right)
\end{equation}
\end{corollary}
{\it Example}

 Given a particle
with  $\textsc{Q}= \mathbb{R}^3$ and kinetic energy $\displaystyle{ T
= \frac{1}{2}( \dot \xi^2 + \dot \eta^2+ \dot \zeta^2)}$.
 \smallskip
 \smallskip
Construct the  field of force capable of generating the
two-parametric family of trajectories defined as intersections of
the two families of surfaces
\begin{equation}\label{BB1}
  f_1=\zeta  = c_1,\qquad f_2=H( \xi, \eta, \zeta)  = c_2 .
\end{equation}
\bigskip

The solution of this problem can easily be derived from Corollary
\ref{SS33}. The vector field $\textbf{v}$,
$\mbox{rot}{\textbf{v}}$ are the following
\[\begin{array}{ll}
 \textbf{v}& = \lambda \left(\dfrac{\partial H
}{\partial \eta}\dfrac{\partial }{\partial \xi}- \dfrac{\partial
H}{\partial \xi}\dfrac{\partial }{\partial \eta} \right) \\  \,
rot \, \textbf{v} & = \dfrac{\partial }{\partial \zeta} \left (
\lambda \dfrac{\partial H }{\partial \xi}\right)\dfrac{\partial
}{\partial \xi} + \dfrac{\partial }{\partial \zeta} \left (
\lambda \dfrac {\partial H}{\partial \eta}\right)\dfrac{\partial
}{\partial \eta} - \mu \dfrac{\partial }{\partial \zeta}
\end{array}\]
Hence the require field of force is such that
\begin{equation}\label{DDD31}
\begin{array}{ll}
\ddot{\xi}=\dfrac{\partial}{\partial \xi}
\left(\dfrac{\lambda^2}{2}\left((\dfrac {\partial H}{\partial \xi}
)^2
+ (\dfrac {\partial H}{\partial \eta} )^2\right)\right)+\lambda\mu\dfrac{\partial H }{\partial \xi}\vspace{0.20cm}\\
\ddot{\eta}=\dfrac{\partial}{\partial \eta}\left(
 \dfrac{\lambda^2}{2}\left((\dfrac {\partial H}{\partial \xi} )^2 +  ( \dfrac
{\partial H}{\partial \eta} )^2\right)
\right)+\lambda\mu\dfrac{\partial H }{\partial \eta}\vspace{0.20cm}\\
\ddot{\zeta}=0,
\end{array}
\end{equation}
where  $ \mu = \dfrac{\partial }{\partial \xi} \Bigl ( \lambda
\dfrac{\partial H }{\partial \xi}\Bigr) + \dfrac{\partial
}{\partial \eta} \Bigl ( \lambda \dfrac{\partial H }{\partial
\eta} \Bigr)$

 In the next section we make use of the
solution of the generalized  Dainelli inverse problem for solving  the Suslov and  generalized
 Joukovski problems.

 \smallskip
 \smallskip

{\it Statement of Suslov's problem}\quad\cite{Sus}

\smallskip
 \smallskip

Given a $N-1$ -parametric family of orbits $f_j=f_j(x)=c_j$ for $j=1,2,\ldots,N-1$ in the configuration
space $\textsc{Q}$ of a holonomic system with $N$ degrees of freedom and
kinetic energy $T=\dfrac{1}{2}||\dot{\textbf{x}}||^2.$ {\it Suslov's problem} is the
problem of determining the potential field of force that
under which any trajectory of the family can be traced by a representative point
of the system.

\smallskip
\smallskip

{\it Solution of Suslov's problem}

\smallskip
\smallskip
We now propose a new solution to the Suslov problem. This solution we have obtained as a special case of the previous solution to the generalized inverse Dainelli problem.

 \begin{proposition}\label{SS}

 Given a mechanical system with a configuration space $\textsc{Q}$  and a kinetic energy $T=\displaystyle\dfrac{1}{2}||\dot{\textbf{x}}||^2$, then the
potential field of force $\textbf{F}=\displaystyle\frac{\partial
U}{\partial{\textbf{x}}},$ capable of generating the given orbits $ f_j (
x) = c_j, $ for $j=1, \ldots, N-1 $ , can be determine from the formula
\begin{equation}\label{D3}
\displaystyle\frac{\partial
U}{\partial{\textbf{x}}}=\displaystyle\frac{\partial
\left(\dfrac{1}{2}||\textbf{v}||^2\right)}{\partial{\textbf{x}}}+\lambda\sum_{j=1}^{N-1}a_{Nj}\displaystyle\frac{\partial
f_j}{\partial{\textbf{x}}},
\end{equation}
if and only if
  \begin{equation}\label{Su}  \lambda
\displaystyle\sum_{j=1}^{N-1}a_{Nj}(x)df_j = dh (f_1, f_2, \ldots,
f_{N-1})
\end{equation}
Clearly if \eqref{Su} holds then  the potential function $U$ is
such that \[U(x)= \frac{1}{2}\Vert \textbf{v}  \Vert^2 + h (f_1,
f_2, \ldots, f_{N-1})\, , \] where $\textbf{v}$ is determined by
the formula \eqref{T1}, and $\lambda$ is an arbitrary function.
\end{proposition}
\begin{proof}
From \eqref{D30} follows that
$\textbf{F}=\displaystyle\frac{\partial U}{\partial{\textbf{x}}},$
if and only if $\lambda
\displaystyle\sum_{j=1}^{N-1}a_{Nj}(x)df_j$ is exact 1-form $dh$.
In view of the relations $ df_j(\textbf{v})=0$ we deduce that $
dh(\textbf{v})=0,$ hence in view of independence of functions
$f_j$ for $j=1,2,\ldots,N-1$ we get $ h=(f_1, f_2, \ldots,
f_{N-1}).$

\end{proof}
{\it Example}(Ermakov's problem).

Given a mechanical system  with configuration space $\textsc{ Q}=
\mathbb{R}^4$ and kinetic energy
$$ T =\left(\frac{1}{2}( m_1(\dot{x}^2_1+\dot{y}^2_1) + m_2( \dot{x}^2_2+\dot{y}^2_2)\right).$$
  \smallskip
Construct the  potential field of force capable of generating the
three-parametric family of trajectories defined as intersections
of the  families of hyper-surfaces
\[f_1=x^2_1+y^2_1=c_1,\qquad f_2=x^2_2+y^2_2=c_2,\qquad
f_3=(x_1-x_2)^2+(y_1-y_2)^2=c_3\]

\begin{corollary}\label{SS22}
Under the assumptions of Proposition \ref{SS} for $N=2$ we obtain
that
\begin{equation}\label{DD3}
\displaystyle\frac{\partial
U}{\partial{\textbf{x}}}=\displaystyle\frac{\partial
\left(\dfrac{1}{2}||\textbf{v}||^2\right)}{\partial{\textbf{x}}}+\lambda\left(\partial_x(\lambda
\partial_xf)+\partial_y(\lambda
\partial_yf\right)\displaystyle\frac{\partial
f}{\partial{\textbf{x}}}.\end{equation} if and only if

\begin{equation}\label{DD4}
\lambda\left(\partial_x(\lambda
\partial_xf)+\partial_y(\lambda
\partial_yf)\right)df=dh(f)\end{equation}
where $\textbf{v}=\left(
-\lambda\partial_yf,\,\lambda\partial_xf\right)$
\end{corollary}

 Another interesting application of the
solution to  the generalized Dainelli problem is the determination
of the solution of the generalized Joukovski problem.

\smallskip
\smallskip

{\it Statement of generalized Joukovski problem}

 \smallskip
\smallskip

Given a $N-1$ -parametric family of orbits $f_j=f_j(x)=c_j$ for $j=1,2,\ldots,N-1$ in the configuration
space $\textsc{Q}$ of a holonomic system with $N$ degrees of freedom and
kinetic energy $T=\dfrac{1}{2}||\dot{\textbf{x}}||^2.$  Assuming complementary that the given trajectories admit a family of orthogonal hyper-surface $S=S(x)=c_N$ then  {\it Generalized  Joukovski problem} is the
problem of determining the potential field of force that
under which any trajectory of the family can be traced by a representative point
of the system.

\smallskip
\smallskip

{\it Solution of Generalized Joukovski problem}

 \smallskip
\smallskip

\begin{proposition}\label{SS11}
Given a mechanical system with a configuration space $\textsc{Q}$  and  kinetic energy $T=\displaystyle\dfrac{1}{2}\|\dot{\textbf{x}}\|^2$, then the
potential field of force $\textbf{F}=\displaystyle\frac{\partial
U}{\partial{\textbf{x}}},$ capable of generating the given orbits $ f_j (x) = c_j, $ for $j=1, \ldots, N-1 $ , which admit a family of orthogonal hyper-surface $S=S(x)=c_N$  can be determine from the formula

\[
 \displaystyle\frac{\partial U}{\partial{\textbf{x}}}=
 \displaystyle\dfrac{\partial\left(\dfrac{\nu}{\sqrt{2}} \|\displaystyle\dfrac{\partial
S}{\partial{\textbf{x}}}\|\right)^2}{\partial{\textbf{x}}}+(
\displaystyle\dfrac{\partial
\dfrac{\nu^2}{2}}{\partial{\textbf{x}}},\,\displaystyle\dfrac{\partial
S}{\partial{\textbf{x}}})\displaystyle\dfrac{\partial
S}{\partial{\textbf{x}}}-\| \displaystyle\dfrac{\partial
S}{\partial{\textbf{x}}}\|^2\dfrac{\partial
\dfrac{\nu^2}{2}}{\partial{\textbf{x}}}\]

if and only if
\begin{equation}\label{SS2}
( \displaystyle\dfrac{\partial
\nu^2}{\partial{\textbf{x}}},\,\displaystyle\dfrac{\partial
S}{\partial{\textbf{x}}})dS-\| \displaystyle\dfrac{\partial
S}{\partial{\textbf{x}}}\|^2d\nu^2=-2dh(f_1,\,f_2,\ldots,\,f_{N-1})
\end{equation}

where $ \nu$ is an arbitrary function. Clearly if \eqref{SS2} holds then the potential function $U$ can be determined as follow
\[ U=\left(\dfrac{\nu}{\sqrt{2}} \|\displaystyle\dfrac{\partial
S}{\partial{\textbf{x}}}\|\right)^2+h(f_1,\,f_2,\ldots,\,f_{N-1}).\]
\end{proposition}

\begin{proof}
In view of the condition that there is an orthogonal hyper-surface to the given trajectories then  the following relations hold
$ \left( \displaystyle\dfrac{\partial
S}{\partial{\textbf{x}}},\,\displaystyle\dfrac{\partial
f_j}{\partial{\textbf{x}}}\right)=0$
for $j=1,2,\ldots,N-1,$ thus
\begin{equation}\label{M1}
\rho G^{-1}\displaystyle\dfrac{\partial
S}{\partial{\textbf{x}}}= \{f_1,\,f_2,\ldots,f_{N-1},\,*\}
\end{equation}
where $ \rho$ is an arbitrary nonzero function and $G^{-1}=(G^{jk})$ is the inverse matrix of the Riemann metric $G$.
 Hence we obtain  that the vector field \eqref{T1} takes the form
\[
 \textbf{v}=\nu G^{-1}\displaystyle\dfrac{\partial
S}{\partial{\textbf{x}}},
\]
where $\nu=\lambda\rho ,$ therefore the 1-form associated to this
vector field is such that
$\sigma=\nu\left(\displaystyle\dfrac{\partial
S}{\partial{\textbf{x}}},\,d\textbf{x}\right)$  consequently
\[ \imath_\textbf{v}d\sigma=d\nu(\textbf{v})dS- dS(\textbf{v})d\nu=\nu( \displaystyle\dfrac{\partial
\nu}{\partial{\textbf{x}}},\,\displaystyle\dfrac{\partial
S}{\partial{\textbf{x}}})dS-\nu| \displaystyle\dfrac{\partial
S}{\partial{\textbf{x}}}\|^2d\nu\]
hence we easily obtain the proof of the proposition (see for more details proof of Theorem \eqref{main}).
\end{proof}

 As a first application of the above proposition we have the following results
\begin{corollary}\label{H1}
If in Proposition \ref{SS11} we suppose that $\nu=\dfrac{d\Phi (S)}{d{S}}$ then the potential function $U$ can be determine as follows
\[ U=\left(\dfrac{1}{\sqrt{2}} \|\displaystyle\dfrac{\partial\Phi (S)}{\partial{\textbf{x}}}\|\right)^2-h_0,\]where $h_0$ is an arbitrary constant.
\end{corollary}
\begin{proof}
Indeed if $\nu=\dfrac{d\Phi (S)}{d{S}}$ then then the 1-form
$\sigma$ associated to vector field $\textbf{v}$ is exact, thus
$\imath_\textbf{v}d\sigma==0,$ consequently (see formula
\eqref{SS2}) $dh=0$. Hence
\[ U=\left(\dfrac{\nu}{\sqrt{2}} \|\displaystyle\dfrac{\partial
S}{\partial{\textbf{x}}}\|\right)^2-h=\left(\dfrac{1}{\sqrt{2}}
\|\displaystyle\dfrac{\partial\Phi
(S)}{\partial{\textbf{x}}}\|\right)^2-h_0.\]
\end{proof}

{\it Example}( Bertrand's Problema  \cite{Bertrand})

 Given a particle
with configuration space $\textsc{ Q}= \mathbb{R}^3$ and kinetic energy
$${ T = \frac{1}{2}( \dot \xi^2 + \dot \eta^2+ \dot
\zeta^2)}.$$
  \smallskip
Construct the  potential field of force capable of generating the
two-parametric family of trajectories defined as intersections of
the two families of surfaces
\[
  f_1=\zeta  = c_1,\qquad f_2=\sqrt{\xi^2+\eta^2}+b\xi  = c_2 .
\]
From \eqref{M1} we obtain that the function $S:$
\[\dfrac{\partial S}{\partial{{\xi}}}=\dfrac{1}{\rho}\dfrac{\eta}{\sqrt{\xi^2+\eta^2}},\quad \dfrac{\partial
S}{\partial{{\eta}}}=\dfrac{1}{\rho}\left(-\dfrac{\xi}{\sqrt{\xi^2+\eta^2}}-b\right).
\]Hence, by choosing $\rho=\eta$ we obtain that $S=\ln (\xi+\sqrt{\xi^2+\eta^2})-(b+1)\ln \eta$ and $\nu=\eta\lambda.$
Clearly the field of force is potential in particular if $\nu=\Phi(S).$ The general solution of this problem we give below.

\smallskip
  \smallskip

{\it Example}

 Given a particle
with configuration space $\textsc{ Q}= \mathbb{R}^3$ and kinetic energy
$${ T = \frac{1}{2}( \dot \xi^2 + \dot \eta^2+ \dot
\zeta^2)}.$$
  \smallskip
Construct the  potential field of force capable of generating the
two-parametric family of trajectories defined as intersections of
the two families of surfaces
\[f_1=xz=c_1,\qquad f_2=yz=c_2\]
From \eqref{M1} we obtain that the function $S=\dfrac{1}{2}(x^2+y^2-z^2),$ thus the condition \eqref{SS2} takes the form
\[\begin{array}{cc}
(x\partial_x\nu^2+y\partial_y\nu^2-z\partial_z\nu^2)(xdx+ydy-zdz)-(x^2+y^2+z^2)d\nu^2\vspace{0.20cm}\\
=-2dh(f_1,f_2).
\end{array}
\]
There are two obvious solutions $ \nu=\nu(x^2+y^2-z^2)$ and
$\nu=z.$ The first solution produces the potential function
$U=U(x^2+y^2+z^2)-h_0,\quad h=h_0$ which coincide with the
solution obtained by Joukovski and the second gives the potential
$U=\dfrac{1}{2}z^4-h_0,\quad h=\dfrac{1}{2}(f^2_1+f^2_2)-h_0.$

 \smallskip
  \smallskip

The following result is a generalization of Theorem \ref{J1}.

\begin{corollary}
If $x^j=C_j=const,$ for $j=1,2,..,N-1$ are the equations of the
$N-1$ parametric family of curves on $\textsc{Q}$, and $x^N=const$ denotes
the family of curves orthogonal to these, then the curves
$x^j=C_j=const$ can be freely described by a particle under the
influence of forces derived from the potential-energy function
$$U=\frac{1}{G_{NN}(x^1,x^2,..,x^N)}\Big(g(x^N)+\sum_{j=1}^{N-1}\int h(x^1,x^2,..,x^{N-1}) \frac{\partial
G_{NN}(x^1,x^2,..x^{N})}{\partial x^j}dx^j\Big)$$ where $h=h(x^1,\,x^2,\,\ldots,x^{N-1})$ and
$g=g(x^N)$ are arbitrary functions.

Clearly, for $N=2$ we obtain the Joukovski theorem given in the
introduction.
\end{corollary}

\begin{proof}
Whereas in the study case $f_j=x^j,\,j=1,2,\ldots N-1$ and $S=x^N$
then from Proposition \ref{SS11} we obtain that the relations hold
\[\begin{array}{ll}\displaystyle\dfrac{\partial 2U}{\partial{{x}^j}}=
 \displaystyle\dfrac{\partial\left(\nu^2
 G^{NN}\right)}{\partial{{x}^j}}-
 G^{NN}\displaystyle\dfrac{\partial
\nu^2}{\partial{{x}^j}},\quad j=1,2,\ldots N-1,\\
\displaystyle\dfrac{\partial 2U}{\partial{{x}^N}}=
 \displaystyle\dfrac{\partial\left(\nu^2 G^{NN}\right)}{\partial{{x}^N}}-\left(\sum_{j=1}^{N-1}G^{Nj}
 \dfrac{\partial\nu^2}{\partial{{x}^j}}\right )
 \end{array}
\]where $G_{kj}=G_{kj}(x^1,\,x^2,\ldots x^N)$ for $k,j=1,2,\ldots N,$
if and only if
\begin{equation}\label{SS3}
( \displaystyle\sum_{j=1}^NG^{jN}\dfrac{\partial
\nu^2}{\partial{x^j}})dx^N-G^{NN}d\nu^2=2dh(f_1,\,f_2,\ldots,\,f_{N-1}),
\end{equation}
assuming that the Riemann metric is orthogonal then

$d\nu^2=-2G_{NN}dh,$ where \,$G^{NN}=\dfrac{1}{G_{NN}}.$

\[\begin{array}{ll}
 \nu^2=2g(x^N)-2\displaystyle\int G_{NN}dh\vspace{0.20cm}\\
 =2g(x^N)-2G_{NN}h+\displaystyle\sum_{j=1}^{N-1}2\displaystyle\int h\dfrac{\partial{G_{NN}}}{\partial x^j}dx^j,
 \end{array}
\]
where $g=g(x^N)$ is an arbitrary function.

Consequently in view of the formula $ U=\dfrac{1}{2}\nu^2G^{NN}+h=\dfrac{\nu^2}{2G_{NN}}+h$ we obtain the proof of the proposition.
\end{proof}

Now we apply Theorem \ref{SS11} to solve the inverse problem which we will call the inverse {\it St{\"a}ckel
problem}.

\smallskip

{\it Statement of inverse  St{\"a}ckel problem}.

\smallskip

Given a $N-1$ -parametric family of orbits
\begin{equation}\label{SS5}
f_\mu=f_{\mu}(x)\equiv\sum_{k=1}^n
\int\frac{\varphi_{k\mu}(x^k)}{\sqrt{K_k(x^k)}}dx^k=c_\mu,\quad
\mu=1,2,...,N-1,
\end{equation}
where
$K_k(x^k)=2\Psi_k(x^k)+2\displaystyle\sum_{j=1}^N\alpha_j\varphi_{kj}(x^k),$\,\,
$\alpha_k,$ for $k=1,2,..N$ are constants, in the configuration
space $\textsc{Q}$ of a holonomic system with $N$ degrees of
freedom and kinetic energy
\begin{equation}\label{SS6}
T=\dfrac{1}{2}||\dot{\textbf{x}}||^2=\dfrac{1}{2}\displaystyle\sum_{j=1}^N\dfrac{(\dot{x}^j)^2}{A^j} ,
\end{equation}
 where
 $A^j=A^j(x)$ for $j=1,2,\ldots,N$ are functions such that
 \[
  \dfrac{\{\varphi_1,\,\varphi_2,\,\ldots,\varphi_{N-1},\,*\}}{\{\varphi_1,\,\varphi_2,\,
  \ldots,\varphi_{N-1},\,\varphi_N\}}=\displaystyle\sum_{j=1}^NA^j\partial_j
\]
 $d\varphi_\alpha=\displaystyle\sum_{k=1}^N\varphi_{k\alpha}(x^k)dx^k,\quad\varphi_{k\alpha}=\varphi_{k\alpha}(x^k),$
for  $k=1,...,N,\, \alpha=1,\ldots,N$ are arbitrary functions.

\smallskip

 The {\it inverse  St{\"a}ckel problem } is the
problem of determining the potential field of force that
under which any trajectory of the family can be traced by a representative point
of the system.

\smallskip

{\it Solution of inverse  St{\"a}ckel problem}

 \smallskip

\begin{proposition}\label{SS7}
Given a mechanical system with a configuration space $\textsc{Q}$
and  kinetic energy \eqref{SS6}, then the potential field of force
$\textbf{F}=\displaystyle\dfrac{\partial
U}{\partial{\textbf{x}}},$ capable of generating the given orbits
\eqref{SS5} is the field of force with potential function
\begin{equation}\label{St1}
U=\nu^2(S)\left(\dfrac{\{
\varphi_1,\,\varphi_2,\,\ldots,\,\varphi_{N-1},\,\Psi\}}{\{
\varphi_1,\,\varphi_2,\,\ldots,\,\varphi_{N-1},\,\varphi_N\}}+\alpha_1\right)-h_0,\end{equation}
where
$S=2\displaystyle\int\displaystyle\sum_{j=1}^N\sqrt{\Psi_k(x^k)+\displaystyle\sum_{j=1}^N\alpha_j\varphi_{kj}(x^k)}\quad
dx^k.$
\end{proposition}
\begin{proof}
In  view of the equality
\[\begin{array}{ll}
\dfrac{\{f_1,\,f_2,\,\ldots,f_{N-1},\,*\}}{\{f_1,\,f_2,\,\ldots,f_{N-1},\,f_N\}}=
 \displaystyle\dfrac{\left|
\begin{array}{ccccc}
 q_1(x^1)d\varphi_1(\partial_1)& \ldots& q_N(x^N)d\varphi_1 (\partial_N) \\
  \vdots & & \vdots \\
   q_1(x^1)d\varphi_{N-1}(\partial_1)
& \ldots & q_N(x^N)d\varphi_{N-1}(\partial_N)\\
\partial_1 &  \ldots & \partial_N \end{array}
         \right|}{\prod_{j=1}^Nq_j(x^j)\{\varphi_1,\,\varphi_2,\,\ldots,\varphi_{N-1},\,\varphi_N\}}\vspace{0.30cm}\\
=\displaystyle\sum_{j=1}^N\left(\dfrac{A^j}{q_j(x^j)}\partial_j\right)=\displaystyle\sum_{j=1}^N\left({A^j}\dfrac{
\partial S}{\partial x^j}\partial_j\right)
       \end{array}
         \]
where $q_j(x^j)=\dfrac{1}{ \sqrt{K_k(x^k)}},$  consequently from
\eqref{M1} and \eqref{SS5} follows that
\[\textbf{v}=-\lambda_N G^{-1}\displaystyle\dfrac{\partial
S}{\partial{\textbf{x}}}=\nu G^{-1}\displaystyle\dfrac{\partial
S}{\partial{\textbf{x}}},\]hence
\[\begin{array}{ll}\|\textbf{v}\|^2=\nu^2\displaystyle\sum_{j=1}^NA^j(K_j(x^j))^2=\nu^2\displaystyle\sum_{k=1}^NA^k\left(2\Psi_k(x^k)+2\displaystyle\sum_{j=1}^N\alpha_j
\varphi_{kj}(x^k)\right)\vspace{0.20cm}\\
=2\nu^2\displaystyle\sum_{k=1}^NA^k\Psi_k(x^k)+2\nu^2\displaystyle\sum_{j=1}^N\alpha_j\displaystyle\sum_{k=1}^NA^k
\varphi_{kj}(x^k)\vspace{0.20cm}\\
= 2\nu^2\left(\dfrac{\{
\varphi_1,\,\varphi_2,\,\ldots,\,\varphi_{N-1},\,\Psi\}}{\{
\varphi_1,\,\varphi_2,\,\ldots,\,\varphi_{N-1},\,\varphi_N\}}+\displaystyle\sum_{j=1}^N\alpha_j\dfrac{\{
\varphi_1,\,\varphi_2,\,\ldots,\,\varphi_{N-1},\,\varphi_j\}}{\{
\varphi_1,\,\varphi_2,\,\ldots,\,\varphi_{N-1},\,\varphi_N\}}\right)\vspace{0.20cm}\\
=.2\nu^2\left(\dfrac{\{
\varphi_1,\,\varphi_2,\,\ldots,\,\varphi_{N-1},\,\Psi\}}{\{
\varphi_1,\,\varphi_2,\,\ldots,\,\varphi_{N-1},\,\varphi_N\}}+\alpha_1\right).
\end{array}
\]
where $d\Psi=\displaystyle\sum_{j=1}^N\Psi_k(x^k)dx^k.$

 On the other hand  if we choose $
\nu=\nu(S)$ then from Corollary \ref{H1} follows the proof of the
proposition.
\end{proof}
\begin{corollary}
If in the previous proposition we suppose that $\nu(S)=1$ and $
\alpha_1=h_0$  then the potential function \eqref{St1} coincide
with the St{\"a}ckel potential \cite{Charlie}.
\end{corollary}
\begin{proof}
Indeed, if $\nu(S)=1$ and $ \alpha_1=h_0$  then \cite{Charlie}
\[U=\dfrac{\{
\varphi_1,\,\varphi_2,\,\ldots,\,\varphi_{N-1},\,\Psi\}}{\{
\varphi_1,\,\varphi_2,\,\ldots,\,\varphi_{N-1},\,\varphi_N\}}=\displaystyle\sum_{k=1}^NA^k\Psi_k(x^k),\]
\end{proof}

 \smallskip
 \smallskip

{\it Statement of Bertrand's problem}

\smallskip

 Let a particle with configuration space $\textsc{Q}=\mathbb{R}^2$
 and kinetic energy $T=\dfrac{1}{2}\left(
 \dot{x}^2+\dot{y}^2\right)$ .  {\it Bertrand's inverse problem} is the problem of determining
 the most general potential field of force capable of generating
 the one-parametric family of conics $f=\sqrt{x^2+y^2}+bx=c,$\,where
 $b$ is the eccentricity.

\smallskip
 \smallskip

{\it Solution of  Bertrand's inverse problem}

\smallskip

 Now we study the
problem of constructing  the potential field of force, which is
capable to generate given conics. We prove the following
proposition.

\smallskip
 \smallskip

\begin{proposition}\, The  potential-energy function $U$ capable of generating a
one-parameter family of conics with eccentricity $b$ is the
function \[\begin{array}{rr} U=a_{-1}H_{-1}(\cos\theta)+K_{-1}\log
r(1+b\cos\theta)\\
+\displaystyle\sum_{j\in
\mathbb{Z}\backslash\{-1\}}a_jr^{j+1}\left(H_j(\cos\theta )+
K_j\frac{(1+b\cos\theta)^{j+1}}{j+1}\right)\end{array}\]if $b\ne
0$
 where $a_j\quad j\in\mathbb{Z},\,K_j$ are real constants and $H_j,\,j\in\mathbb{Z}$ are
solutions of the Heun equations with singularities at the points
$$0,\,1\,,\frac{1+b}{2b},\infty$$
and with the exponents
$$ (0, \frac{j+3+b(j+1)}{2b}); \,\, (0, j-\frac{j+3+b(j+1)}{2b});
\, (0, j+1); (-1-j, 1-j)$$ and
\[ U=\dfrac{\Psi (\cos\theta
)}{r^2}-\dfrac{2}{r^2}\displaystyle\int h(r)dr\] if $b=0,$ where
$\Psi=\Psi (\cos\theta)$ and $ h=h(r)$ are arbitrary functions.
\end{proposition}
\begin{proof}

{F}rom Proposition \ref{SS22} follows that the require potential
field of force exist if and only if
\[
\left(\dfrac{x}{\sqrt{x^2+y^2}}+b\right)\dfrac{\partial\lambda^2}{\partial
x}+\dfrac{y}{\sqrt{x^2+y^2}}\dfrac{\partial\lambda^2}{\partial y}+
\dfrac{2\lambda^2}{r}=2\dfrac{\partial h}{\partial f}
\]
By introducing the polar coordinates $x=r\cos \theta, \,\, y= r
\sin \theta$, we find that previous condition condition takes the
form
\[
( 1+b\cos \theta) \frac{\partial \lambda^2}{\partial r}-
\frac{b\sin \theta}{r} \frac{\partial \lambda^2}{\partial \theta}
+ \frac{2 \lambda^2}{r} = 2 \frac{\partial h}{\partial f}
\]
 or, equivalently,
 \begin{equation}\label{B0}
 ( 1  +b\tau )
\frac{\partial \lambda^2}{\partial r}+ \frac{b (1-\tau^2) }{r}
\frac{\partial \lambda^2}{\partial \tau } + \frac{2 \lambda^2}{r}
= 2 \frac{\partial h}{\partial f}
\end{equation}
where $ f  = r(1+b\tau ), \,\, \tau = \cos \theta .$
%\sum_{\substack{m=1\\m\ne{l}}}
 We embark now upon the study of the case when $ b\ne 0$ and $h$ is such that
\begin{equation}\label{B2}
h(f) = \nu_{-1}\ln |f| + \sum_{\substack{ j\in \mathbb{Z}\\j\ne{-1}}} \nu_j \frac{f^{j+1}}{j+1},
\end{equation}
 where $\nu_j,\,j\in \mathbb{Z},$ are
real constants, and $\lambda$ is determined in such a way that

\begin{equation}\label{B3}
\lambda^2 = \sum_{j\in \mathbb{ Z}} \psi_j (r)H_j(\tau)
\end{equation}

It is clear that the series \eqref{B2} and \eqref{B3}  are formal
series.

By inserting  \eqref{B2} and \eqref{B3} into \eqref{B0} we obtain
$$
\displaystyle\sum_{j\in  \mathbb{Z}} \left(( 1+b\tau ) \frac{d \psi_j
(r)}{d r}H_j(\tau)+ \frac{\psi_j (r)}{r}\left(b (1-\tau^2)
\frac{dH_j(\tau)}{d \tau } + 2\right)-2\nu_jr^j( 1+b\tau
)^j\right) =0
$$
This equation holds if $$  \psi_j (r)  = a_jr^{j+1}, \quad\nu_j =
-a_j K_j  $$ for $ j \in  \mathbb{Z} $ and we determine $H_j$  as a
solution to the equation
\begin{equation}\label{B4}
  b  (1-\tau^2) H_j'(\tau) + \bigl  ( (j+1)b \tau
+j+3 \bigr )H_j(\tau) +2K_j  ( 1+b\tau )^j =0 \end{equation} for $
j  \in  \mathbb{Z}.$

The general solution of  this    equation is \[\begin{array}{cc}
H_j(\tau) = \xi_j (\tau) \left ( C_j -
\dfrac{2K_j}{b}\displaystyle\int \dfrac{ (1+
b\tau)^j}{(1-\tau^2)\xi_j (\tau) }d\tau \right)\vspace{0.30cm}\\
\xi_j (\tau) = (1-\tau)^{\dfrac{j+1}{2}+ \dfrac{j+3}{2b}}
(\tau+1)^{\dfrac{j+1}{2}-\dfrac{j+3}{2b}}\end{array} \]
 where
$C_j,\,j\in\mathbb{Z}$ are arbitrary constants.

\smallskip

 Under these conditions, the required potential-energy function $U$ results in  the form $$ U(r, \tau) =
\frac{1}{2}\lambda^2 ( 1+b^2+2b\tau) - h(f) = \sum_{j \in \mathbb{Z}}a_jU_j(r, \tau)$$ where \[\begin{array}{cc}
 U_j (r, \tau) & =
\dfrac{1}{2}r^{j+1} H_j (\tau) ( 1+b^2+2b \tau) + \dfrac{K_j}{j+1}
f^{j+1} \quad \mbox{if} \,\, j \neq -1 \vspace{0.20cm}\\
 U_{-1}(r, \tau) & = \dfrac{1}{2}H_{-1} (\tau) (1+b^2+2b\tau) + K_{-1}\ln
|f |. \end{array}\] We will study the subcase when $b=1$
separately from the subcase when  $b\neq 1$.

\smallskip

If $b=1,$ $$ U(r, \tau) = \lambda^2 ( 1+ \tau)-h(f) =   \sum_{j
\in \mathbb{Z}}a_j U_j (r, \tau).$$ where
\[\begin{array}{cc}
 U_j (r, \tau) & =
r^{j+1}(1-\tau)^{j+2} \left ( C_j - 2 K_j \displaystyle\int \dfrac
{( 1+ \tau)^j}{(1-\tau)^{j+3}} d \tau \right) + \dfrac{K_j}{j+1}
f^{j+1}, \quad if \,\, j \neq -1 \vspace{0.20cm}\\  U_{-1} (r,
\tau)  & = (1-\tau)\left ( C_{-1} - 2 K_{-1} \displaystyle\int
\dfrac {d\tau}{(1-\tau)^{2}( 1+ \tau)} \right) + K_{-1}\ln |f|.
\end{array}\]

Easily verifies that
$$
U_{-2} = \frac {C_{-2}}{r} -2 \frac {K_{-2}}{r} \Bigl (
\displaystyle\int \dfrac {d\tau}{(1+\tau)^{2}( 1-\tau)} + \dfrac
{1}{1+\tau} \Bigr) = \frac{C_{-2}}{r} + \frac{K_{-2}}{r} g( \tau)
$$ where $ g( \tau) =
\ln \sqrt{\displaystyle\dfrac {1-\tau}{1+\tau}}.$ Therefore, if
$b=1$,
$$U(r,\tau) = \dfrac{a_{-2}C_{-2}}{r} + \dfrac{a_{-2}K_{-2}g(\tau)}{r}+
\sum_{\substack{ j\in \mathbb{Z}\\j\ne{-2}}}a_jU_j(r, \tau) $$ If $ b \neq
1$,  $ b \neq 0$, it is easy to prove that
\[\begin{array}{cc}
H_{-2} (\tau) = \dfrac{(1-\tau)^{\dfrac
{1-b}{2b}}}{(1+\tau)^{\dfrac{1+b}{2b}}} C_{-2} - \dfrac{2
K_{-2}}{(b \tau +1)(1-b^2)}\vspace{0.30cm} \\
U_{-2} (r, \tau) =\dfrac{H_{-2}}{2r} (1+b^2+2b\tau) - \dfrac{
K_{-2}}{r(b \tau +1)} = \dfrac{2 K_{-2}}{r(b^2-1)} +
\dfrac{C_{-2}}{r} G(\tau)
\end{array}\]
 where
$$ G( \tau) = \dfrac{1}{2}\sqrt { \Bigl ( \dfrac
{1-\tau}{1+\tau}\Bigr)^{\dfrac {1}{b}} \dfrac {1}{1-\tau^2} }
(1+b^2+2b\tau)$$ Under these conditions, the potential function
$U$ takes the form $$U(r, \tau) = \dfrac{a_{-2}C_{-2}}{r} G(\tau)
+ \dfrac{2a_{-2}K_{-2}}{r(b^2-1)}+ \sum_{\substack{ j\in \mathbb{Z}\\j\ne{-2}}}a_jU_j(r, \tau) $$

Summarizing the above computations,  we deduce that if $b \neq 0$
the function $U$ is represented as follows:
$$U(r, \tau) = \frac{\alpha}{r}+ \frac{\beta(\tau)}{r}+ \sum_{\substack{ j\in \mathbb{Z}\\j\ne{-2}}}a_jU_j(r, \tau)  
   $$ where $\alpha$ is a
constant and $\beta=\beta(\tau)$ is a certain function.

\smallskip

 If $b=0$,
then $f=r$ and condition \eqref{B0} takes the form
$$
\partial_r \lambda^2 + 2 \frac{\lambda^2}{r}= 2 \partial_f h(f)\, .
$$
Therefore,  $$
r^2 \lambda^2 = 2 \displaystyle\int r^2  \partial_r h(r)dr +
2\Psi (\tau),
$$ which rearranged results in the expression:
$$\lambda^2 = \dfrac{2}{ r^2}\displaystyle \int r^2
\partial_r h(r)dr + \dfrac{2\Psi (\tau)}{ r^2} = 2h(r)- \dfrac{4}{
r^2} \int h(r)rdr + \dfrac{2\Psi (\tau) }{ r^2}$$ where $ \Psi$ is
an arbitrary function.

Hence, $$ U (r, \tau) = \dfrac{\Psi(\tau)}{r^2} - \dfrac{2}{r^2}
\displaystyle\int h(r)dr.$$

\smallskip

To end the proof of the claim will establish the relationship
between the functions $H_j$ for $j\in\mathbb{Z}$ and the solutions
of Heun´s equations.

The canonical form  of Heun's general equation will be taken as
\cite{Ron}
\begin{equation}\label{B5} 
\dfrac{d^2x}{dz^2} + \biggl (
\dfrac{\gamma}{z}+ \dfrac{\delta}{z-1} + \dfrac{\epsilon}{z-a}
\biggr)\dfrac{dx}{dz} + \dfrac{\alpha \beta z - B}{z(z-1)(z-a)} x
= 0 ,    
 \end{equation}
 here $x$ and $z$ are
regarded as complex variables and $\alpha, \beta, \gamma, \delta,
\epsilon, a, b$ are parameters, generally complex and arbitrary,
with the only condition that  $a \not = 0, 1$. The first five
parameters are linked by the relation $\alpha + \beta +1 = \gamma
+ \delta + \epsilon$.

The equation is, therefore, of the Fuchsian type , with
regular singularities at the points $z=0, 1, a, \infty$. The
exponents at these singularities are, respectively, $(0, 1-\gamma
), (0, 1-\epsilon), (0, 1- \delta), (\alpha, \beta)$

\bigskip

Now  we establish the relation  between equation \eqref{B4} and
Heun's equation.

By fulfilling the replacement  $${z} = \dfrac{1}{2}(\tau+1)$$ we
can  easily obtain the following representation for \eqref{B4}:
$${z}({z}-1) \dfrac{dH_j}{dz}+
\dfrac{1}{2b} \bigl  ( (1+b-2bz)(j+1)+2\bigr )H_j(z) -
\dfrac{K_j}{b}( 1+b-2bz)^j=0 $$

By differentiating and fulfilling  some straightforward
calculations, we deduce that the functions $$F_j ( {z}) = \Bigl
(z(z-1) \dfrac{dH_j}{dz }+ \dfrac{1}{2b} \bigl (
(1+b-2bz)(j+1)+2\bigr )H_j(z)\Bigr) ( 1+b-2b  {z})^{-j} \quad ( j
\in  \mathbb{Z})$$ are first integral of the Heun equation:

 \begin{equation}\label{B7}
 \begin{array}{ll}
 \dfrac{d^2H_j}{dz^2}  + \left( \dfrac{1-a(1+j)
-\dfrac{1}{b}}{z} + \dfrac{a(1+j)+ \dfrac{1}{b}-j}{z -1} +
\dfrac{-j}{z-a} \right)\dfrac{dH_j}{dz}\vspace{0.20cm}\\ +
\dfrac{(j^2-1)z - \dfrac{j}{b} - a (j^2-1)}{z ( z-1) (z -a)}
H_j(z) = 0
 \end{array}
 \end{equation}
 where $ a=\dfrac{1+b}{2b}$.

By comparison with classical Heun´s equation, we obtain:
\[\begin{array}{ll}
\gamma_j = 1-\dfrac{1+b}{2b}(1+j) -\dfrac{1}{b}, \quad
\delta_j=\dfrac{1+b}{2b}(1+j)+ \dfrac{1}{b}-j,
\vspace{0.20cm}\\
\alpha_j\beta_j  = j^2-1 = -( 1+ \epsilon_j) (2-\gamma_j
-\delta_j),\quad\quad \epsilon_j= -j \\ B_j  =  \dfrac{j}{ b} +
\dfrac{1+b}{2b} (j^2-1)= -a (2-\gamma_j -\delta_j)-
(1-\gamma_j)\epsilon_j
\end{array}\]

Evidently,  when the given conics are parabolas then in \eqref{B7}
we have the confluence of singularities. In fact, in this case
$b=1$ so $a=1,$ and as a consequence the Heun equation is
transformed into hypergeometric differential equation

 $$
\dfrac{d^2H_j}{dz^2}  + \biggl ( \dfrac{-1-j}{z} + \dfrac{2-j}{z
-1}\biggr)\dfrac{dH_j}{dz} + \dfrac{(j^2-1)(z-1) - j} {z ( z-1)^2}
H_j(z) = 0, $$ for $ j\in\mathbb{Z}.$ Which completed the proof of
Proposition.
\end{proof}

\subsection*{Acknowledgments}
This work was partly supported by the Spanish Ministry of
Education through projects DPI2007-66556-C03-03,
TSI2007-65406-C03-01 "E-AEGIS" and Consolider CSD2007-00004
"ARES".

%\section*{References}

\end{document}